\documentclass{amsart}
\usepackage{amssymb}   
\usepackage{amsmath}
\usepackage{amsthm}
\usepackage[all,dvips]{xy}
\usepackage[]{pstricks,pst-node,pst-plot}
\usepackage{graphicx}
\usepackage{pstcol,pst-char}
\def\a={{\buildrel a \over =}}
\def\na={{\buildrel a \over \neq}}

\def\PP{{\mathbb P}}

\def\Tcal{{\mathcal T}}

\def\C{{\mathcal C}}

\def\S{{\mathcal S}}

\def\tH{\tilde{H}}
\def\tT{\tilde{\mathcal T}}
\def\bM{\bar{M}}

\def\Lg{\Lambda}
\def\tb{T_b}

\newtheorem{theorem}{Theorem}[section]
\newtheorem{lemma}[theorem]{Lemma}
\newtheorem{proposition}[theorem]{Proposition}
\newtheorem{corollary}[theorem]{Corollary}

\newtheorem{definition-lemma}[theorem]{Definition-Lemma}

\theoremstyle{definition}
\newtheorem{definition}[theorem]{\bf Definition}
\newtheorem{example}[theorem]{\bf Example}

\theoremstyle{remark}
\newtheorem{remark}[theorem]{\bf Remark}

\def\vandaag{\number\day\space\ifcase\month\or
 januari\or februari\or  maart\or  april\or mei\or juni\or  juli\or
 augustus\or  september\or  oktober\or november\or  december\or\fi,
\number\year}
\def\today{\ifcase\month\or
 Jan\or Febr\or  Mar\or  Apr\or May\or Jun\or  Jul\or
 Aug\or  Sep\or  Oct\or Nov\or  Dec\or\fi
 \space\number\day, \number\year}

\begin{document}

\title[Rational Correspondences between Moduli Spaces of Curves]
{Rational Correspondences between Moduli Spaces \\
of Curves defined by Hurwitz Spaces}
\author{Gerard van der Geer $\,$}
\address{Korteweg-de Vries Instituut, Universiteit van
Amsterdam, Postbus 94248, 1090 GE Amsterdam,  The Netherlands}
\email{G.B.M.vanderGeer@uva.nl}
\author{$\,$ Alexis Kouvidakis}
\address{Department of Mathematics, University of Crete,
GR-71409 Heraklion, Greece}
\email{kouvid@math.uoc.gr}
\begin{abstract} By associating to a curve $C$ and a $g^1_{d}$ the so-called
trace curve and reduced trace curve
we define two rational maps $\phi$ and $\hat{\phi}$ 
from the Hurwitz space of admissible covers
of genus $g=2k$ and degree $d=k+1$ to moduli spaces $\bM_{g'}$ and 
$\bM_{\hat{g}}$. We study the induced map of the 
divisor class group of $\bM_{g'}$
and $\bM_{\hat{g}}$ to the divisor class group of $\bM_{g}$.
\end{abstract}
\maketitle
\begin{section}{Introduction}\label{sec:intro}
Hurwitz spaces of admissible covers give rise to maps and
correspondences between moduli spaces of curves. In this paper
we study two examples of this. The Hurwitz space $\bar{H}_{g,d}$
in question is the space of admissible covers of even genus $g=2k$ 
and degree $d=k+1$. The general curve of genus $g=2k$ possesses
finitely many linear systems of projective dimension $1$ and degree
$d=k+1$. The Hurwitz space $\bar{H}_{g,d}$ is thus a generically finite
cover of the moduli space of stable curves $\bar{M}_g$.

In \cite{GF} Farkas constructed for odd $g=2k+1$ a rational map 
$\bar{M}_{g} \dashrightarrow \bar{M}_{g'}$ with 
$g'= 1+ {2k+2 \choose k}\frac {k}{k+1}$ by associating to a generic curve $C$
the curve $W^1_{k+2}$ in ${\rm Pic}^{k+2}(C)$ and calculated the induced action on the divisor class group. 
As an application he showed the
upper bound $\sigma(g) < 6 + 16/(g-1)$ for the slope $\sigma(g)$ of the movable
cone of $\bar{M}_g$ for odd genera $g$. 

In this paper we deal with the even genus case $g=2k$ 
and use a completely different construction to define a rational map.
To a general curve $C$ of genus $g$ together with a $g^1_d$, say $\gamma$, 
with $d=k+1$ we associate the so-called {\sl trace curve} $T=T_{C,\gamma}$ 
defined by
$$
T_{C,\gamma} = \{ (x,y) \in C \times C : \gamma \geq p+q \},
$$
the locus of ordered pairs $(p,q)$ contained in the fibers of $\gamma$.
By extending this definition to a suitable open part of the Hurwitz space
we obtain a rational map $\phi: \bar{H}_{g,d} \dashrightarrow\bar{M}_{g'}$
with $g'=5k^2-4k+1$ and it fits into a diagram
\begin{displaymath}
\begin{xy}
\xymatrix{
\bar{H}_{2k,k+1}   \ar@{-->}[r]^{\phi} \ar[d]^{p} & \bar{M}_{g'}\\
\bar{M}_{2k} \\
}
\end{xy}
\end{displaymath}
Note that the ratio $g'/g$ for the genera of the trace curve and the original
curve is much lower than the ratio
in the construction of Farkas.

The main body of this paper is devoted to calculating the induced
action $p_* \phi^*$ on divisor class group of $\bM_{g'}$. 
The trace curve carries a natural involution and by dividing the
trace curve by it we obtain the reduced trace curve. This yields
a similar rational map $\hat{\phi} : \bar{H}_{2k,k+1} 
\dashrightarrow \bM_{\hat{g}}$ with $\hat{g}=(5k-2)(k-1)/2$ and we calculate the
induced map on the divisor class group.

The reduced trace curve has gonality $\leq k(k+1)/2$ and
carries a correspondence that gives rise to an endomorphism $e$ 
of its Jacobian satisfying $(e-1)(e+k-2)=0$.
It is an interesting question to determine further properties of trace curves.

As in the Farkas paper the map $p_* \phi^*$  sends the ample cone 
of $\bar{M}_{g'}$ to the movable cone of $\bar{M}_g$ and we obtain 
in this way a bound on the movable slope of the form $\sigma(g) < 6 + 20/g$
for $g$ even.
But, as we shall show, by viewing the Hurwitz space $\bar{H}_{2k,k+1}$ 
as a correspondence between $\bM_{g}$ and $\bM_{0,6k}$, with $\bM_{0,6k}$ the moduli space of 
stable $6k$-pointed rational curves, one can obtain the slightly better bound $\sigma(g) < 6+18/(g+2)$. 

Besides the rational maps $\phi$ and $\hat{\phi}$ defined 
by the trace curve and its quotient we
also have a rational map $\chi$ of $\bar{H}_{2k,k+1}$ 
to a moduli space of semi-abelian varieties defined by the Prym variety of
the trace curve over the reduced trace curve and a variant 
given by a quotient of the Jacobian of the reduced trace curve. 
These maps deserve further study.

Maps between moduli spaces, like the Torelli map and the Prym map,
can be important tools for a better understanding of moduli spaces.
Since the rational maps and correspondences constructed here 
involve the geometry of the algebraic curve in a natural way it 
is not unreasonable to expect the same for these correspondences.

\end{section}
\begin{section}{The Trace Curve of a $g^1_d$}
Let $C$ be a smooth projective curve of genus $g$ and let $\gamma$ be a
$g^1_d$, that is, a linear system of degree $d$ and projective dimension $1$. 
To the pair $(C,\gamma)$ one can associate an algebraic curve, 
called the {\sl trace curve} and defined by
$$
T_{\gamma} =T_{C,\gamma} :=\{ (p,q) \in C\times C\, :  \gamma \geq p+q \}.
$$
Here the notation $\gamma \geq p+q$ means that there is an effective divisor
in $\gamma$ containing the divisor $p+q$. In the following we shall assume 
that the linear system $\gamma$ is without base points. The trace curve can 
have singularities. More precisely we have the following result,
see Lemma 5.1 in \cite{GK}.

\begin{lemma}
For a base point free $\gamma $ the trace curve $T_{\gamma}$ is a smooth except
for possible singularities at points where both $p$ and $q$ are ramification points of
$\gamma$.  A ramification point $p$ of order $m$ of $\gamma$ gives rise
to an ordinary singular point $(p,p)$ of order $m-1$. 
A point $(p,q) \in T_{\gamma}$
with $p\neq q$ and $p$ and $q$ both simple ramification points is 
a simple node of $T_{\gamma}$.
\end{lemma}

It follows from the above description of the trace curve that
if $(p,q)$ is a smooth point of $T_{\gamma}$,
then it is a ramification point of the first (resp.\ second) 
projection of $T_{\gamma}$ on $C$ if and only if $q$ 
(resp.\ $p$) is a ramification point of $\gamma$.

We recall the following lemma from our \cite{GK}, Lemma 5.2.

\begin{lemma}
Let $\gamma$ be a base point free $g^1_d$ with all branch points simple
except one with arbitrary ramification. Then $T_{\gamma}$ is irreducible.
\end{lemma}

For general $(C,\gamma)$ the trace curve $T_{C,\gamma}$ 
is thus a smooth irreducible curve
of genus
$$
g'=(g-1)(2d-3)+(d-1)^2.
$$
Indeed, the class of the line bundle $O(T_{\gamma})$ 
defined by the trace curve $T_{\gamma}$ on $C \times C$ 
equals $p_1^*L\otimes p_2^*L \otimes O(-\Delta)$ with $p_i$ ($i=1,2$)
the two projections, $L$ the line bundle defining $\gamma$ 
and $\Delta$ the class of the diagonal, as one easily checks by
restricting to horizontal and vertical fibres, hence
globally on $C \times C$. The homology class of $T_{\gamma}$ is 
then $d(F_1+F_2)-[\Delta]$ with $F_i$ the fibre of $p_i$. 
The adjunction formula implies the formula for the genus $g^{\prime}$.

The trace curve $T_{\gamma}$ possesses an involution $\iota$ induced
by interchanging the two factors of $C \times C$. The fixed points
of $\iota$ are exactly the intersection points of $T_{\gamma}$ with
the diagonal and these are the points $(p,p)$ with $p$ a ramification
point of $\gamma$. We define the {\sl reduced trace curve} 
$\hat{T}_{\gamma}=\hat{T}_{C,\gamma}$
as the quotient curve $T_{\gamma}/\iota$.

We are interested in the case that $g=2k$ is even and $d=k+1$. A generic curve $C$ of genus $2k$ has only finitely
 many $g^1_d$ with $d=k+1$, namely
$N=N(k)={2k \choose k+1}/k$. For a generic $\gamma$ on
such a smooth curve the geometric genus of the trace curve $T_{\gamma}$ equals
$$
g^{\prime} = 5 k^2-4k+1 \, ,
$$
while the geometric genus of the reduced trace curve $\hat{T}_{\gamma}$
equals
$$
\hat{g}= \frac{(5k-2)(k-1)}{2} \, .
$$
\begin{remark}
Note that by construction the reduced trace curve possesses a morphism
of degree $k(k+1)/2$ to ${\PP}^1$ defined by sending $p+q$ to $\gamma(p)=
\gamma(q)$. So the gonality is much lower than $[(\hat{g}+3)/2]$. 
\end{remark}

\begin{example}
For $k=2$ the trace curve of a curve $C$ of genus $4$ with a $g^1_3$
has genus $13$ while the reduced 
trace curve has genus $4$ and is isomorphic to $C$.
\end{example}
The construction of the trace curve can be done in families. This defines
a morphism $\phi: H_{g,d} \to  M_{g'}$ with $H_{g,d}$ the Hurwitz scheme
of simple covers of the projective line ${\PP}^1$ of degree $d=k+1$ and
genus $g=2k$. Here simple means that the fibres of $\gamma$ always have 
at least $d-1$ points. We thus get correspondences
\begin{displaymath}
\begin{xy}
\xymatrix{
H_{2k,k+1} \ar[r]^{\phi} \ar[d]^{p}& {M_{g'}} && 
H_{2k,k+1} \ar[r]^{\hat{\phi}} \ar[d]^{p} & M_{\hat{g}}\\
M_{2k} &&& M_{2k} \\
}
\end{xy}
\end{displaymath}

\begin{example}
Let $k=3$ and let 
$C$ be a general curve of genus $6$. According to [ACGH, p.\ 218] the curve
is birational to a plane sextic with four nodes. The five
$g^1_4$ are given by the four linear systems obtained by the lines
through a node and by the conics through all four nodes. The reduced trace
curve $\hat{T}_{\gamma}$ associated to such a $g^1_4$ is of genus $13$ 
and carries a fixed point free involution:
if $p_1+p_2+p_3+p_4$ is a divisor from the $g^1_4$ and $p_1+p_2$
belongs to the reduced trace curve then the corresponding point
is $p_3+p_4$. This involution is fixed point free for general $(C,\gamma)$.
So we get a curve $T_{\gamma}^{\prime}$ of genus $7$ as the 
quotient of the reduced trace curve.  This curve is a trigonal curve
and the Prym variety of the \'etale double cover $\hat{T}_{\gamma} \to 
T_{\gamma}^{\prime}$ is known to be isomorphic to ${\rm Jac}(C)$.
So up to isogeny ${\rm Jac}(\hat{T}_{\gamma})$ is a product of
${\rm Jac}(T_{\gamma}^{\prime})$ and ${\rm Jac}(C)$. Our map
$H_{6,4} \to M_{13}$ factors through a map $H_{6,4} \to M_{7}$
and is dominant on the trigonal locus ${\mathcal T}_7$ 
in $M_{7}$.  Note that both
$M_{6}$ (or $H_{6,4}$) and the trigonal locus ${\mathcal T}_{7}$ have 
dimension $15$. It seems that $H_{6,4} \to {\mathcal T}_7$
is birational.
\end{example}
\smallskip

The reduced trace curve carries a correspondence:

\begin{proposition}
For  $(C,\gamma)$ in $H_{2k,k+1}$ the
reduced trace curve $\hat{T}$ possesses a correspondence 
that induces an endomorphism $e$ of ${\rm Jac}(\hat{T})$
satisfying a quadratic equation $(e-1)(e+k-2)=0$ in 
${\rm End}({\rm Jac}(\hat{T}))$.
\end{proposition}
\begin{proof}
This follows from a result of Kanev, cf.\
\cite{Kanev}, (Prop.\ 5.8), p.\ 265.
The correspondence is given by
$$
D:=\{ (p+q,r+s) \in \hat{T}^2: \gamma \geq p+q+r+s \} \, .
$$
This induces an endomorphism $e$ of ${\rm Jac}(\hat{T})$
that decomposes ${\rm Jac}(\hat{T})$; define an abelian subvariety
$A=A_e$ of ${\rm Jac}(\hat{T})$ as the image of the
endomorphism $1-e$. It follows from the result of
Kanev loc.\ cit.\ that $A$ is isogenous (even isomorphic)
to ${\rm Jac}(C)$. 
\end{proof}

That ${\rm Jac}(\hat{T})$ contains an isogenous image of ${\rm Jac}(C)$
can be seen as follows.
The embedding $\rho: \hat{T} \to {\rm Sym}^2(C)$ induces a map
$\rho^*: {\rm Pic}^0({\rm Sym}^2(C)) \to  {\rm Pic}(\hat{T})$.
Now we have an isomorphism ${\rm Pic}^0(C) \to {\rm Pic}^0({\rm Sym}^2(C))$
given by associating to the divisor class $a-b$ the divisor class
$C_a-C_b$ with $C_p$ the image of the map $C \to {\rm Sym}^2(C)$
that sends $q$ to $p+q$. On the other hand we have a map
$z: {\rm Pic}(\hat{T}) \to {\rm Pic}^0(C)$ by associating to $t_1-t_2$
with $t_i=p_i+q_i$ the divisor $p_1+q_1-p_2-q_2$, that is the
image of $t_1-t_2$ under $p_{1*} \sigma^*$ with $\sigma: T \to \hat{T}$ the
natural map and $p_1: T \to C$ the projection. The composition of
$$
{\rm Pic}^0(C) \to {\rm Pic}^0(\hat{T}), 
\qquad a-b \mapsto C_a\cdot \hat{T} -C_b \cdot \hat{T} 
$$
with $p_{1*}\sigma^*$ is $k-1$ on ${\rm Pic}^0(C)$. 
Hence ${\rm Pic}^0(C)$
maps to an  abelian subvariety of ${\rm Pic}^0(\hat{T})$ and the quotient
is an abelian variety of dimension $\bar{g}=(5k-1)(k-2)/2$.
We thus find a map
$$ 
{\hat{\chi}}: H_{2k,k+1} \to {\mathcal A}_{(5k-1)(k-2)/2}, 
\quad\text{given by $(C, \gamma) \mapsto 
{\rm Jac}(\hat{T}_{\gamma})/{\rm Jac}(C)$},
$$
where ${\mathcal A}_{\bar{g}}$ denotes a moduli space
of polarized abelian varieties of dimension $\bar{g}$.

\end{section}
\begin{section}{The Action of the Correspondence on Divisors}\label{sec: actioncorresp}
The Hurwitz space $H_{d,g}$ is a smooth irreducible scheme 
and is compactified by the space of admissible covers $\bar{H}_{d,g}$. 
We can view it as a stack or orbifold, but $\bar{H}_{d,g}$ is not normal. 
We normalize and get a smooth stack $\tH_{d,g}$ of which $H_{d,g}$ 
can be considered as an open dense subspace. For more on this
normalization we refer to \cite{GK}.

Since $\bM_{g'}$ (resp.\ $\bM_{\hat{g}}$) is a smooth stack and
$\tH_{d,g}$ is smooth the map $\phi $  viewed as a rational map
$\bar{H}_{g,d} \dasharrow \bM_{g'}$ has locus of indeterminacy of
codim $\geq 2$.
We thus get maps
$$
\phi^* : {\rm Pic}(\bM _{g'}) \to {\rm Pic}(\tH_{d,g}), \qquad
\hat{\phi}^*: {\rm Pic}(\bM _{\hat{g}}) \to {\rm Pic}(\tH_{d,g}).
$$
When $g=2k$ and $d=k+1$ the natural map $p: \tH_{d,g} \to \bM_g$ is a
generically finite map and we studied in \cite{GK} the behaviour of
the induced map  $p_*: {\rm Pic}(\tH_{d,g}) \to {\rm Pic}(\bM_g)$.

One of the purposes of this paper is to study the composite map
$$
\alpha = p_* \phi^*: {\rm Pic}(\bM _{g'}) \to {\rm Pic}(\bM_g)
$$
and the similar map
$$
\hat{\alpha}= p_{*} \hat{\phi}^*: {\rm Pic}(\bM _{\hat{g}})
\to {\rm Pic}(\bM_g)\, .
$$
In \cite{GK} (Prop. 3.1 and 4.1) we determined the boundary divisors
in $\tH_{d,g}$ which do not map to zero under $p_*$. These are: a divisor
$E_0$ which maps dominantly on $\Delta_0$,  divisors $E_{j,c}$
for $1\leq j \leq  k$ and $0\leq c \leq [j/2]$, that map dominantly
to $\Delta_j$, and divisors $E_2, E_3$, each mapping dominantly to a
divisor in $\bM_g$ that intersects $M_g$. The general point of $E_2$ and
$E_3$ represents a curve which has as stable model a smooth genus $g$ curve.
To study the map $\alpha$ (resp.\ $\hat{\alpha})$ we thus may restrict
ourselves to studying the trace curve (resp.\ reduced trace curve)
for admissible covers in the smooth open substack $\tilde{H}$
of $\tH_{g,d}$
$$
\tH:=H_{d,g} \cup (\cup_{j,c} E_{j,c}) \cup E_0 \cup E_2 \cup E_3.
$$
We shall keep the notation $p:\tH \to \bM_g$ for the natural map.
\end{section}

\begin{section}{extending the Trace Curve}
In order to study divisors on $\tH $ it will suffice to look
at one-dimensional families of admissible covers with 
general member in $H_{g,d}$,
and their associated trace curves. Therefore we study in this section the extension of
the trace curve over $1$-dimensional base curves $B$ in~$\tH$.

Over the Hurwitz scheme $\tH$ we have a universal curve $\C$.
The general cover has $b=6k$ branch points.
The curve $\C$ fits in the following basic diagram
\begin{equation} \label{basicdiagram}
\begin{xy}
\xymatrix{
 \bM_{0,6k+1} \ar[d]_{\varpi}  & \C \ar[d]^{\pi} \ar[l]_{\;\;\;\; \gamma} 
&   \\
 \bM_{0,6k}   &   \tH \ar[l]^{q}   &   \\
}
\end{xy}
\end{equation}
where $q$ is the map that associates to
an admissible cover $C\to P$ the genus $0$ curve $P$ together with the $6k$
branch points.

We now assume that $B$ is a $1$-dimensional smooth base (disk or
the spectrum of a discrete valuation ring). Over $B$ we have the
pull back of the universal curve $\C$ and we can restrict the basic diagram
(\ref{basicdiagram}) to $B$. We shall define the trace curve $\Tcal$
as the closure of the locus of points $(a,b)$ of $\C \times_B \C$
with $a\neq b$ and $a$ and $b$ in the same fibre of $\gamma:
\C \to \bM_{0,6k+1}$.  The fiber of  $\C \times_B\C$
over a point $h \in B$ consists of the products of the various 
components of the curve $\C_h$. 
The fiber $T_h$ of the trace curve over $h$ lies in the product of 
components of $\C_h$ which map, by the map $\gamma$,
to the same rational component of the fiber of the map $\varpi$ 
over the point $q(h)$.
We shall carry out this construction locally. Note that either
both $a$ and $b$ are smooth points of the fibre $\C_{h}$ or both are
singular points. We start with the case of smooth points.

{\bf Case 1: pairs of smooth points.}  
Assume that $a \neq b$ are smooth points of the curve $\C_h$ with $\gamma(a)=
\gamma(b)$. We denote by $\sigma$ a local coordinate on $B$, by $u$
a local coordinate on $\bM_{0,6k+1}$, 
and by $x,y$ (resp.\ $x',y'$) local coordinates on $\C$ at $a$ (resp.\ $b$)
so that $\pi$ at $a$ (resp.\ $b$) is given by $x=\sigma$ (resp.\
$x'=\sigma$) and the map $\gamma$ to $\bM_{0,b+1}$ by $y=u^m$ (
resp.\ by $u=y'$) with $m=1$ or $m=2$ depending on whether $a$
is a ramification point. 
(Since we assume the cover is simple at most one of the smooth points
$a,b$ is a ramification point and if so we assume it is $a$.)
Then the equations for the trace
curve (as a family over $B$) around the point $(a,b)$ with $a\neq b$ are given by
$$
x=\sigma, \, x'=\sigma, \, y=(y')^m.
$$
For $(a,b)$ with $a=b$ and $m=2$ the equations of the trace curve are given by
$$
x=\sigma, \, x'=\sigma,\, y+y'=0\, .
$$
In both cases the corresponding system defines locally a 
smooth family of curves with smooth central fiber.

\bigskip
\noindent
{\bf Case 2: pairs of singular points.}
If the point $a$ is a singular (nodal) point of the curve $\C_h$, 
then any $b$ with $\gamma (a)=\gamma (b)$ is also a singular point of $\C_h$.
In our case we are interested in pairs of singular points $(a,b)$ 
with local equations of one of the following types:

\smallskip
\noindent
{\em Type 1:}
In this case $a \neq b$. The local equation at $a$ of the map 
$\pi: \C \to B$ is $xy=\sigma^m $ 
(in the $x,y, \sigma $  coordinate system as above)
and at $b$ it is $x'y'=\sigma^m $ 
(in the $x',y', \sigma $ coordinate system).
The local equation at $a$ of the map $\gamma : \C \to \bM_{0,6k+1}$ is
of the form $x=u,\, y=v$ and at $b$ it is  $x'=u,\, y'=v$. 
Then the local equations for the trace curve (as a family over $B$) 
at the point $(a,b)$ are given by 
$xy=\sigma ^m, \, x'y'=\sigma ^m,\,  x=x', y=y'$, i.e.\ by
$$
xy=\sigma ^m, \, x=x',\, y=y' .
$$
The last two equations define an intersection of hyperplanes 
and then the first implies that the family has an $A_{m-1}$
singularity at the point $(a,b)$, which we may resolve by 
inserting a chain of $(-2)$ curves of length $m-1$.

\smallskip
\noindent
{\em Type 2:}
In this case $a=b$. The local equation at $a=b$ of the map 
$\pi: \C \to B$ is $xy=\sigma$,
(in the $x,y, \sigma $ coordinate system).  
The local equation at $a=b$ of the map $\gamma: \C \to \bM_{0,6k+1}$ is
of the form $x^m=u,\; y^m=v$. Then the local equations for the trace curve 
(as a family over $B$) at the point $(a,b)$ are given
(in the  $x,y,x',y', \sigma $ coordinate system) by
$$
xy=\sigma , \quad x'y'=\sigma ,\quad \frac{x^m-{x'}^m}{x-x'} =0,
\quad \frac{y^m-{y'}^m}{y-y'} =0\, .
$$
But note that for $\sigma \neq 0$ the last two equations define the 
same locus (because of the first two equations). But for $\sigma =0$
they define the locus of points $(x,0,x',0,0)$  with $(x^m-{x'}^m)/(x-x') =0$ 
(which is the trace curve of the map $x^m=u$ in the
$xx'$-plane) plus the locus of points $(0,y,0,y',0)$
with $(y^m-{y'}^m)/(y-y') =0$ 
(which is the trace curve of the map $y^m=v$ in the $yy'$-plane). The point
$(0,0,0,0,0)$ is a singular point (for $m \geq 3$)  
of the family of trace curves. 
We perform a small blow up (inside the fiber product of curves) 
by setting: $ux'-vx=0,\, uy-vy'=0$. 
The proper transform of the trace curve by the blow up is given by the equations
$$
xy=\sigma , \quad x'y'=\sigma ,\quad ux'-vx=0,\quad uy-vy'=0, \quad 
\frac{u^m-v^m}{u-v} =0 \, .
$$
The last equation gives $u=\omega ^i v$, $i=1, \ldots, m-1$, with 
$\omega $ a primitive $m$-th root of unity. Therefore the
trace curve intersects the exceptional line at the $m-1$ points $[\omega^i,1]$. In the neighborhood of this point the trace curve
is given by the equations
$$
xy=\sigma,\quad x'=\omega ^i x,\quad 
y'=\omega^{-i}y,\quad [u,v]=[\omega^i,1], \quad i=1, \ldots, m-1\, .
$$
This defines locally a smooth family with nodal central fiber.

\smallskip
\noindent
{\em Type 3:}
In this case $a\neq b$. 
The local equation at $a$ of the map $\pi : \C \to B$ is $xy=\sigma  $,
(in the $x,y, \sigma $ coordinate system) and at $b$ is $x'y'=\sigma $.  
The local equation at $a$ of the map
$\gamma: \C \to \bM_{0,6k+1}$ is of the form $x^2=u,\, y^2=v$ and at $b$ 
it is ${x'}^2=u,\, {y'}^2=v$.
Then the local equations for the trace curve (as a family over $B$) 
at the point $(a,b)$ are given (in the  $x,y,x',y', \sigma $
coordinate system) by
$$
xy=\sigma , \;\; x'y'=\sigma ,\;\; x^2-{x'}^2  =0,\;\; y^2-{y'}^2  =0\, .
$$
We blow up as before and we find that the proper transform of the trace curve by the blow up is given by the equations
$$
xy=\sigma , \;\; x'y'=\sigma ,\;\; ux'-vx=0,\;\; uy-vy'=0, \;\; u^2-v^2 =0 \, .
$$
The last equation gives $u=\pm v$. Therefore the trace curve intersects the exceptional line at the two points $[\pm 1,1]$.
In the neighborhood of these points the trace curve is given by the equations
$$
xy=\sigma,\;\; x'= \pm x,\;\;y'=\pm y,\;\; [u,v]=[\pm 1,1] \, .
$$
This defines locally a smooth family with nodal central fiber.

\smallskip
\noindent
{\em Type 4:}
In this case $a\neq b$.  
The local equation at $a$ of the map $\pi : \C \to B$ is $xy=\sigma$,
(in the $x,y, \sigma $ coordinate system) and at $b$ is $x'y'=\sigma ^m $.  
The local equation at $a$ of the map
$\gamma: \C \to \bM_{0,6k+1}$ is of the form $x^m=u,\, y^m=v$ and at 
$b$ it is $x'=u,\, y'=v$.  Then the local equations for
the trace curve (as a family over $B$) at the point $(a,b)$ are given 
(in the  $x,y,x',y',\sigma $ coordinate system) by
$$
xy=\sigma , \quad x'y'=\sigma ^m ,\quad x'=x^m,\quad y'=y^m  \, .
$$
This defines locally a smooth family of curves with nodal central fiber.

\smallskip
\noindent
{\bf Conclusion.} By performing the small blow-ups at the pairs of points of
type $2$ and type $3$ we created a (singular) nodal model ${\mathcal T}'$
over $B$. By resolving the singularities (of type $A_m$) we obtain a 
smooth model  $\tau: \tT  \to B$ 
of the trace curve, a nodal family of curves with smooth total space.

\end{section}
\begin{section}{The Geometry of the Trace Curve}\label{Geometry}
In our study of the divisors in $\tH $ we shall need to know
the shape of the trace curve near a point
of the divisors $E_0, E_{j,c}, E_2$ and $E_3$ in $\tH$.
We may assume that the limit point is a generic point of a
component of one of these divisors. The reader can find the description
of the generic admissible cover over any of these divisors in
our paper \cite{GK}. For each of these cases we explicitly carry out the
construction done in the preceding section. 

In the following figures \ref{fiberoverE0}, \ref{fiberoverE2}, 
\ref{fiberoverE3} and \ref{fiberoverEjc},
on the left we show the fiber $\C_h$ of $\C \to \tH$ over a generic point 
$h$  of the boundary components $E_0$, $E_2$,
$E_3$ and $E_{j,c}$ respectively. 
On the right we show the corresponding fiber $\tT_h$ of 
the smooth model of the family of the trace curves as constructed 
as in the preceding section and its first projection 
$\eta : \tT_h \to \C_h$.

Figure \ref{fiberoverE0} corresponds to the case
where $h$ is a general point of $E_0$.
The admissible cover on the left is described as follows:
it consists of a main component, $C$ is a curve of genus $2k-1$, and
rational  curves $R_1, \ldots, R_{k-1}$ and $S$. This maps to
a rational curve consisting of two components $\PP_1$ and $\PP_2$.
The map from the curve $C$ to $\PP_1$ has degree $k+1$. 
The components $R_i$ map 
isomorphically to $\PP_2$ and the map from $S$  to $\PP_2$
has degree $2$. 
At all the intersection points of the above components,  
the admissible cover has ramification degree $1$. 
The  rational curve $\PP_1$ contains $6k-2$ branch points and 
$\PP_2$ contains $2$ branch points.

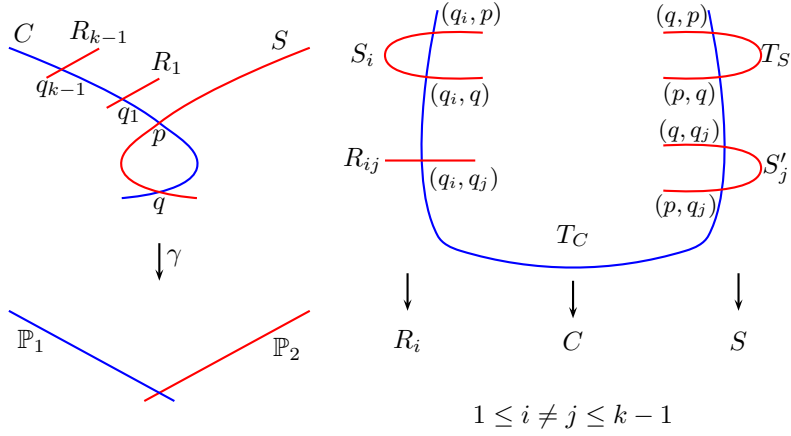
\begin{figure}
\begin{center}
\begin{pspicture}(-5,-3)(5,3)
\psline[linecolor=red](-3.2,-2.2)(-1,-1) 
\psline[linecolor=blue](-2.8,-2.2)(-5,-1) 
\pscurve[linecolor=blue](-5,2.5)(-3,1.5)(-2.5,0.9)(-3.5,0.5) 
\pscurve[linecolor=red](-1,2.5)(-3,1.5)(-3.5,0.9)(-2.5,0.5) 
\psline{->}(-3,-0.1)(-3,-0.6)
\rput(-2.8,-0.3){$\gamma $}
\psline[linecolor=red](-4.5,2.1)(-3.8,2.49)
\psline[linecolor=red](-3.7,1.7)(-3,2.09)
\rput(-4.8,2.7){$C$}
\rput(-1.4,2.6){$S$}
\rput(-1.3,-1.5){${\PP}_2$}
\rput(-4.7,-1.4){${\PP}_1$}
\rput(-3,1.3){$p$}
\rput(-3,0.4){$q$}
\rput(-2.9,2.3){$R_1$}
\rput(-3.8,2.7){$R_{k-1}$}
\rput(-3.4,1.6){$q_1$}
\rput(-4.3,1.95){$q_{k-1}$}

\pscurve[linecolor=blue](0.7,3) (0.7,0) (1,-0.2) (4,-0.2) (4.3,0) (4.3,3)
\rput(2.5,-2.4){$1 \leq i\neq j\leq k-1$}
\rput(2.5,0){$T_C$}
\pscurve[linecolor=red] (1.3,2.7) (0,2.4) (1.3,2.1)
\rput(-0.3,2.4){$S_i$}
\rput(1.15,2.95){\small $(q_i,p)$}
\rput(1,1.87){\small $(q_i,q)$}
\psline[linecolor=red] (0,1) (1.2,1)
\rput(-0.3,1){$R_{ij}$}
\rput(1.05,0.75){\small $(q_i,q_j)$}
\pscurve[linecolor=red] (3.7,2.7) (5,2.4) (3.7,2.1)
\rput(5.2,2.4){$T_S$}
\rput(3.95,2.92){\small $(q,p)$}
\rput(4.05,1.88){\small $(p,q)$}
\pscurve[linecolor=red] (3.7,1.2) (5,0.9) (3.7,0.6)
\rput(5.2,0.9){$S'_j$}
\rput(4.05,1.4){\small$(q,q_j)$}
\rput(4,0.4){\small $(p,q_j)$}
\psline{->}(0.3,-0.5)(0.3,-1)
\rput(0.3,-1.4){$R_i$}
\psline{->}(2.5,-0.6)(2.5,-1.1)
\rput(2.5,-1.4){$C$}
\psline{->}(4.7,-0.5)(4.7,-1)
\rput(4.7,-1.4){$S$}
\end{pspicture}
\end{center}
\caption{The fiber of $\C$ over a point of $E_0$ and the 
corresponding trace curve}
\label{fiberoverE0}
\end{figure}

The trace curve on the right has the following properties:
\begin{enumerate}
\item All the singular pairs $(a,b)$ of points are of type 1 with $m=1$.
\item The curves $T_C$ and  $T_S$ are the trace curve of the maps 
$C \to \PP_1$ and $S \to \PP_2$ respectively.
The curves $S_i$ (resp. $S'_j$) are produced by taking pairs of 
points from the components $R_i$ and $S$
(resp.\ $S$ and $R_j$). The curves $R_{ij}$ are produced by 
taking pairs of points from the components
$R_i$ and $R_j$.      
\item 
 The curves $R_{ij}$, $S_i$, $T_S$ and $S'_j$ are all rational curves.
\item The map  $ S_i \to R_i$ is 2:1 and  the  maps  $R_{ij} \to R_i$, and $T_s, S'_j \to S$  are all isomorphisms.
\end{enumerate}

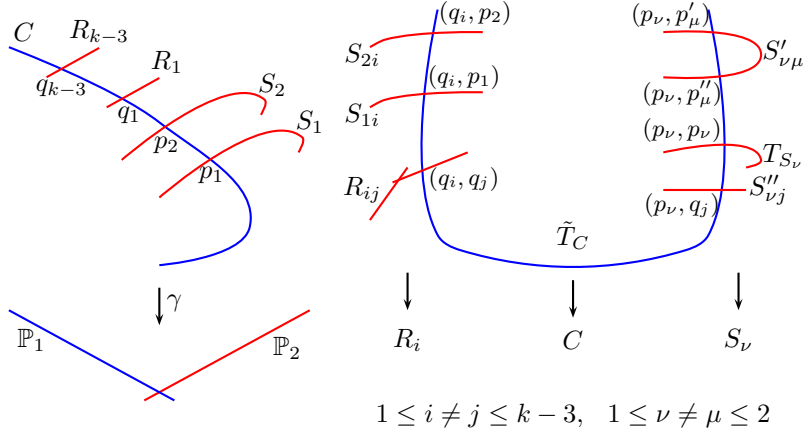
\begin{figure}
\begin{center}
\begin{pspicture}(-5,-3)(5,3)
\psline[linecolor=red](-3.2,-2.2)(-1,-1) 
\psline[linecolor=blue](-2.8,-2.2)(-5,-1) 
\pscurve[linecolor=blue](-5,2.5)(-3,1.5)(-1.8,0.1)(-3,-0.4) 
\pscurve[linecolor=red](-1.65,1.6) (-1.6,1.8) (-3.5,1) 
\pscurve[linecolor=red](-1.15,1.1) (-1.1,1.3) (-3,0.5) 
\psline{->}(-3,-0.7)(-3,-1.2)
\rput(-2.8,-0.9){$\gamma $}
\psline[linecolor=red](-4.5,2.1)(-3.8,2.49)
\psline[linecolor=red](-3.7,1.7)(-3,2.09)
\rput(-4.8,2.7){$C$}
\rput(-1.5,2){$S_2$}
\rput(-1,1.5){$S_1$}
\rput(-1.3,-1.5){${\PP}_2$}
\rput(-4.7,-1.4){${\PP}_1$}
\rput(-2.9,1.2){$p_2$}
\rput(-2.3,0.8){$p_1$}
\rput(-2.9,2.3){$R_1$}
\rput(-3.8,2.7){$R_{k-3}$}
\rput(-3.4,1.6){$q_1$}
\rput(-4.3,1.95){$q_{k-3}$}

\pscurve[linecolor=blue](0.7,3) (0.7,0) (1,-0.2) (4,-0.2) (4.3,0) (4.3,3)
\rput(2.5,-2.4){$1 \leq i\neq j\leq k-3$, $\;\;1\leq \nu \neq \mu \leq 2$}
\rput(2.5,0){$\tilde{T}_C$}
\pscurve[linecolor=red] (1.3,2.7) (0,2.6) (-0.2,2.5)
\rput(-0.3,2.4){$S_{2i}$}
\rput(1.2,2.95){\small $(q_i,p_2)$}
\pscurve[linecolor=red] (1.3,1.9) (0,1.8) (-0.2,1.7)
\rput(-0.3,1.6){$S_{1i}$}
\rput(1.05,2.1){\small $(q_i,p_1)$}
\psline[linecolor=red] (0.1,0.7) (1.1,1.1)
\psline[linecolor=red]  (0.3,0.9) (-0.2,0.2)
\rput(-0.3,0.6){$R_{ij}$}
\rput(1.05,0.75){\small $(q_i,q_j)$}
\pscurve[linecolor=red] (3.7,2.7) (5,2.4) (3.7,2.1)
\rput(5.3,2.4){$S'_{\nu\mu}$}
\rput(3.8,2.92){\small $(p_{\nu},p'_{\mu})$}
\rput(3.95,1.88){\small $(p_{\nu},p''_{\mu})$}
\pscurve[linecolor=red] (3.7,1.1) (5,1) (4.8,0.9)
\rput(5.3,1.05){$T_{S_{\nu}}$}
\rput(3.95,1.4){\small$(p_{\nu},p_{\nu})$}
\psline[linecolor=red](3.7,0.6) (4.8,0.6)
\rput(5.1,0.6){$S''_{\nu j}$}
\rput(3.95,0.4){\small $(p_{\nu},q_j)$}
\psline{->}(0.3,-0.5)(0.3,-1)
\rput(0.3,-1.4){$R_i$}
\psline{->}(2.5,-0.6)(2.5,-1.1)
\rput(2.5,-1.4){$C$}
\psline{->}(4.7,-0.5)(4.7,-1)
\rput(4.7,-1.4){$S_{\nu}$}
\end{pspicture}
\end{center}
\caption{The fiber of $\C$ over a point of $E_2$ and the corresponding trace curve}
\label{fiberoverE2}
\end{figure}

Figure \ref{fiberoverE2} corresponds to the case where $h$ is a 
general point of $E_2$. The admissible cover on the left is 
described as follows: it consists of a curve $C$ of genus $2k$ 
and rational curves $R_1, \ldots, R_{k-3}$, $S_1$ and $S_2$. 
The map from the curve $C$ to $\PP_1$ has degree $k+1$.
The components $R_i$ map isomorphically to
$\PP_2$ and the maps from $S_1$ and $S_2$ to $\PP_2$ have degree $2$. 
The admissible cover has ramification degree $1$ at the points 
$q_1, \ldots, q_{k-3}$ and ramification degree $2$ at the points
$p_1$ and $p_2$. The rational curve $\PP_1$ contains $6k-2$ branch points 
while $\PP_2$ contains $2$ branch points.

The trace curve on the right has the following properties:
\begin{enumerate}
\item The pairs $(q_i,q_j)$ are of type 1, with $m=2$; 
the pairs $(p_1,p_1)$ and $(p_2,p_2)$ are of type 2,
with $m=2$; the pairs $(p_1,p_2)$ and $(p_2,p_1)$ are of type 3; 
the pairs $(q_i, p_{\nu})$ and $(p_{\nu}, q_j)$ are of type 4, with $m=2$.
\item The curve $\tilde{T}_C$ is the normalization of the trace curve of 
the map $C \to \PP_1$.
The curves $T_{S_{\nu}}$ are the trace curves of the maps $S_{\nu} \to \PP_2$.
The curves $S_{i1}$ (resp. $S_{i2}$)  are produced by taking pairs of 
points from the components $R_i$ and $S_1$ (resp.\ $R_i$ and $S_2$). 
The curves $S''_{\nu j}$ are obtained by taking pairs of points from 
the components $S_{\nu}$ and $R_j$.
The curves $S'_{\nu\mu}$  are obtained by taking pairs of points 
from the components $S_{\nu}$ and $S_{\mu}$.
The curves $R_{ij}$ are obtained by taking pairs of points from the 
components $R_i$ and $R_j$.
\item The curves $S_{1i}$, $S_{2i}$, $R_{ij}$, 
$S'_{\nu\mu}$, $T_{S_{\nu}}$ and $S''_{\nu j}$ are all rational curves.
 \item The maps $S_{1i}, S_{2i} \to R_i$ are $2:1$ 
and the maps $R_{ij} \to R_i$ and
$S'_{\nu\mu}, T_{S_{\nu}}, S''_{\nu j} \to S_{\nu}$ are all isomorphisms. 
The $(-2)$ curve  which joins $R_{ij}$ with $\tilde{T}_C$ 
contracts to the points $q_i$.
\end{enumerate}

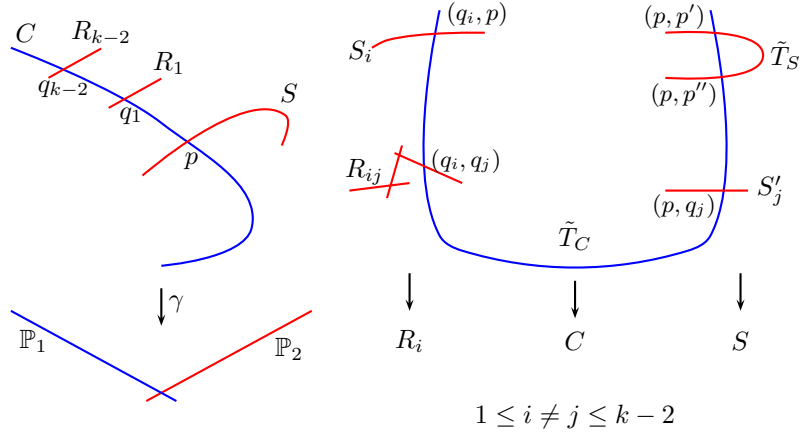
\begin{figure}
\begin{center}
\begin{pspicture}(-5,-3)(5,3)
\psline[linecolor=red](-3.2,-2.2)(-1,-1) 
\psline[linecolor=blue](-2.8,-2.2)(-5,-1) 
\pscurve[linecolor=blue](-5,2.5)(-3,1.5)(-1.8,0.1)(-3,-0.4) 
\pscurve[linecolor=red](-1.4,1.2) (-1.35,1.6) (-3.25,0.8) 
\psline{->}(-3,-0.7)(-3,-1.2)
\rput(-2.8,-0.9){$\gamma $}
\psline[linecolor=red](-4.5,2.1)(-3.8,2.49)
\psline[linecolor=red](-3.7,1.7)(-3,2.09)
\rput(-4.8,2.7){$C$}
\rput(-1.3,1.9){$S$}
\rput(-1.3,-1.5){${\PP}_2$}
\rput(-4.7,-1.4){${\PP}_1$}
\rput(-2.6,1){$p$}
\rput(-2.9,2.3){$R_1$}
\rput(-3.8,2.7){$R_{k-2}$}
\rput(-3.4,1.6){$q_1$}
\rput(-4.3,1.95){$q_{k-2}$}

\pscurve[linecolor=blue](0.7,3) (0.7,0) (1,-0.2) (4,-0.2) (4.3,0) (4.3,3)
\rput(2.5,-2.4){$1 \leq i\neq j\leq k-2$}
\rput(2.5,0){$\tilde{T}_C$}
\pscurve[linecolor=red] (1.3,2.7) (0,2.6) (-0.2,2.5)
\rput(-0.35,2.45){$S_i$}
\rput(1.2,2.95){\small $(q_i,p)$}
\psline[linecolor=red] (0.1,1.1) (1,0.7)
\psline[linecolor=red]  (0.2,1.2) (0,0.5)
\psline[linecolor=red]  (0.3,0.7) (-0.5,0.6)
\rput(-0.3,0.85){$R_{ij}$}
\rput(1.05,1){\small $(q_i,q_j)$}
\pscurve[linecolor=red] (3.7,2.7) (5,2.4) (3.7,2.1)
\rput(5.3,2.4){$\tilde{T}_S$}
\rput(3.82,2.92){\small $(p,p')$}
\rput(3.95,1.88){\small $(p,p'')$}
\psline[linecolor=red](3.7,0.6) (4.8,0.6)
\rput(5.1,0.6){$S'_j$}
\rput(3.95,0.4){\small $(p,q_j)$}
\psline{->}(0.3,-0.5)(0.3,-1)
\rput(0.3,-1.4){$R_i$}
\psline{->}(2.5,-0.6)(2.5,-1.1)
\rput(2.5,-1.4){$C$}
\psline{->}(4.7,-0.5)(4.7,-1)
\rput(4.7,-1.4){$S$}
\end{pspicture}
\end{center}
\caption{The fiber of $\C$ over a point of $E_3$ and the corresponding trace curve}
\label{fiberoverE3}
\end{figure}

Figure \ref{fiberoverE3} corresponds to the case where $h$ is a 
general point of $E_3$.
The admissible cover on the left is described as follows: 
the curves $R_1, \ldots, R_{k-2}$ and $S$ are rational curves. 
The curve $C$ is a curve of genus $2k$. 
The components $R_i$ map isomorphically to
$\PP_2$ and the map from $S$ to $\PP_2$ have degree $3$. 
The map from the curve $C$ to $\PP_1$ has degree $k+1$.
The admissible cover has ramification degree $1$ at the points 
$q_1, \ldots, q_{k-2}$ and ramification degree $3$ at the point  $p$. 
The $\PP_1$ contains $6k-2$ branch points and the  $\PP_2$ contains 
$2$ branch points.

The trace curve on the right has the following properties:
\begin{enumerate}
\item The pairs $(q_i,q_j)$ are of type 1, with $m=3$; 
the pair $(p,p)$ is of type 2, with $m=3$;  the
pairs $(q_i,p)$ and $(p,q_j)$ are of type 4, with $m=3$.
\item The curves $\tilde{T}_C$ and  $\tilde{T}_S$ are the normalizations 
of the trace curves of the maps $C \to \PP_1$
and  $S \to \PP_2$ respectively. 
The curves $S_i$ (resp. $S'_j$)  are produced by taking pairs of points from
the components $R_i$ and $S$ (resp. $S$ and $R_i$). 
The curves $R_{ij}$ are produced by taking pairs of points
from the components  $R_i$ and $R_j$.
\item The curves $S_i$, $R_{ij}$, $\tilde{T}_S$ and $S'_j$ are all 
rational curves.
\item The map $S_i \to R_i$ is 3:1, the map $R_{ij} \to R_i$ is an 
isomorphism, the map $\tilde{T}_S \to S$ is $2:1$
and the map $S'_j \to S$ is an  isomorphism. 
The chain of $(-2)$ curves of length 2 which joins the $R_{ij}$ with
$\tilde{T}_C$ contracts to the point $q_i$.
\end{enumerate}

\begin{figure}
\begin{center}
\begin{pspicture}(-5,-3)(5,3)
\psline[linecolor=red](-3.7,-2.2)(-1.5,-1) 
\psline[linecolor=blue](-3.3,-2.2)(-5.5,-1) 
\psline[linecolor=red](-3.7,0.8)(-1.5,2) 
\psline[linecolor=blue](-3.3,0.8)(-5.5,2) 
\psline{->}(-3.5,-0.7)(-3.5,-1.2)
\rput(-3.3,-0.9){$\gamma $}
\psline[linecolor=red](-5.2,1.6)(-4.5,1.99)
\psline[linecolor=red](-4.4,1.2)(-3.7,1.59)
\psline[linecolor=blue](-1.8,1.6)(-2.5,1.99)
\psline[linecolor=blue](-2.6,1.2)(-3.3,1.59)
\rput(-5.5,2.2){$C_1$}
\rput(-1.5,2.2){$C_2$}
\rput(-1.8,-1.5){${\PP}_2$}
\rput(-5.2,-1.4){${\PP}_1$}
\rput(-3.45,0.65){$p$}
\rput(-3.8,1.8){$S_1$}
\rput(-4.4,2.2){$S_{k-j+c}$}
\rput(-4.05,1.05){\small $p_1$}
\rput(-5,1.4){\small $p_{k-j+c}$}
\rput(-3.2,1.8){$R_1$}
\rput(-2.6,2.2){$R_{c}$}
\rput(-2.9,1.05){\small $q_1$}
\rput(-2.05,1.45){\small $q_{c}$}
\pscurve[linecolor=blue](-0.5,1) (1.9,0) (2.2,-0.4)(2.5,0) (2.8,0.4) (3.1,0) (5.5,-1) 
\rput(5.5,-1.3){$\tilde{T}_{C_1}$}
\pscurve[linestyle=dotted, linecolor=blue](-0.8,0.7) (1.6,-0.3) (1.9,-0.7)
\pscurve[linecolor=blue](-1.1,0.4) (1.3,-0.6) (1.6,-1.0)
\rput(-1.35,0.4){$C_{1\rho}$}
\rput(-0.7,0){\small $(p_{\lambda},q_{\rho})$}
\rput(5.7,0){\small $(q_{\nu},p_{\mu})$}
\rput(0.95,0.8){\small $(p_{\lambda},p)$}
\rput(0.95,-0.8){\small $(p,q_{\rho})$}
\rput(4.05,0.8){\small $(q_{\nu},p)$}
\rput(4.05,-0.8){\small $(p,p_{\mu})$}
\psline{->}(1,-1.1) (0.6,-1.5)
\rput(0.3,-1.8){$C_1$} 
\psline{->}(1,1.1) (0.6,1.5)
\rput(0.3,1.8){$S_{\lambda}$}
\psline{->}(4,-1.1) (4.4,-1.5)
\rput(4.7,-1.8){$C_2$}
\psline{->}(4,1.1) (4.4,1.5)
\rput(4.7,1.8){$R_{\nu}$}
\pscurve[linestyle=dotted, linecolor=blue](3.1,0.7) (3.4,0.3) (5.8,-0.7)
\pscurve[linecolor=blue](3.4,1.0) (3.7,0.6) (6.1,-0.4)
\rput(6.4,-0.4){$C'_{\nu 1}$}
\pscurve[linecolor=red](-0.5,-1) (1.9,0) (2.2,0.4) (2.5,0) (2.8,-0.4) (3.1,0) (5.5,1)
\rput(5.5,1.3){$\tilde{T}_{C_2}$}
\pscurve[linestyle=dotted, linecolor=red](-0.8,-0.7) (1.6,0.3) (1.9,0.7)
\pscurve[linecolor=red](-1.1,-0.4) (1.3,0.6) (1.6,1.0)
\rput(-1.35,-0.4){$C_{\lambda 2}$}
\pscurve[linestyle=dotted, linecolor=red](3.1,-0.7) (3.4,-0.3) (5.8,0.7)
\pscurve[linecolor=red](3.4,-1.0) (3.7,-0.6) (6.1,0.4)
\rput(6.4,0.4){$C'_{2\mu}$}
\rput(2.5,-2.6){$1 \leq \lambda \neq \mu \leq k-j+c$, $\;\; 1 \leq \nu \neq \rho \leq c$}
\psline{->}(2.2,-1)(2.2,-1.5)
\psline{->}(2.8,-1)(2.8,-1.5)
\rput(2.2,-2){$R_{\nu}$}
\rput(2.8,-2){$S_{\lambda}$}
\psline[linecolor=blue] (2.2,0.3) (1.9,1.1)
\psline[linecolor=blue] (1.9,0.9) (2.3,1.8)
\psline[linestyle=dotted, linecolor=blue] (2.3,1.6) (1.9,2.1)
\psline[linecolor=blue] (1.9,1.9) (2.3,2.6)
\rput(1.6,2.4)  {$R_{\nu \rho}$}
\psline[linecolor=red] (2.8,0.3) (3.1,1.1)
\psline[linecolor=red] (3.1,0.9) (2.7,1.8)
\psline[linestyle=dotted, linecolor=red] (2.7,1.6) (3.1,2.1)
\psline[linecolor=red] (3.1,1.9) (2.7,2.6)
\rput(3.4,2.4) {$S_{\lambda \mu}$}
\end{pspicture}
\end{center}
\caption{The fiber of $\C$ over a point of $E_{j,c}$ and the corresponding trace curve}
\label{fiberoverEjc}
\end{figure}
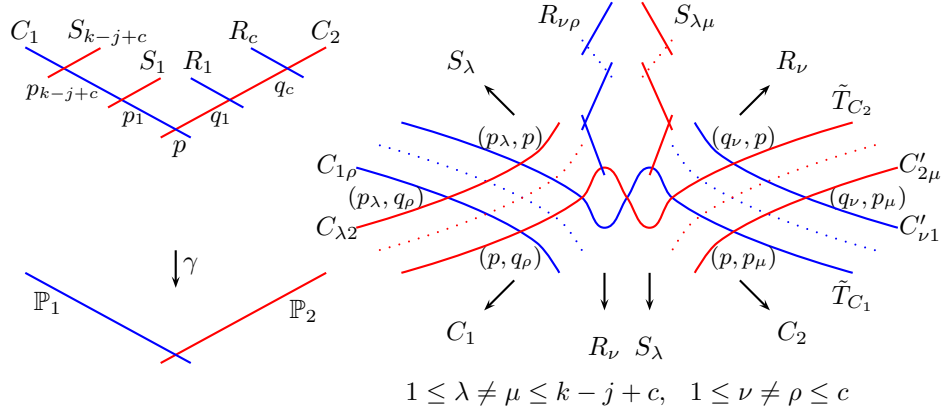

Figure \ref{fiberoverEjc} corresponds to the case where $h$ is a 
general point of $E_{j,c}$. The admissible cover on the left is 
described as follows: The curves $R_1, \ldots, R_c$ and 
$S_1, \ldots, S_{k-j+c}$ are rational curves. 
The curve $C_1$ has genus $2k-j$ and the curve $C_2$ has genus $j$. 
The curve $R_{\nu}$ (resp.\ $S_{\lambda}$) map isomorphically  
to $\PP_1$ (resp.\ $\PP_2$).  The map from the curve $C_1$ to $\PP_1$
has degree $k+1-c$ and  the map from the curve $C_2$ to $\PP_2$ 
has degree $j+1-c$. The $\PP_1$ contains $6k-3j$
branch points and the  $\PP_2$ contains $3j$ branch points.
The trace curve on the right has the following properties:
\begin{enumerate}
\item The pairs $(p_{\lambda}, q_{\rho})$,  
$(q_{\nu},p_{\mu})$, $(p_{\lambda},p_{\mu})$ and
$(q_{\nu}, q_{\rho})$  are of type 1, with $m=j+1-2c$; 
the pair $(p,p)$ is of type 2, with $m=j+1-2c$;
the pairs $(p_{\lambda},p)$, $(q_{\nu},p)$, $(p,p_{\mu})$ 
and $(p,q_{\rho})$ are of type 4, with $m=j+1-2c$.
\item The curves $\tilde{T}_{C_1}$ and $\tilde{T}_{C_2}$ 
are the normalizations of the trace curves of the maps
$C_1 \to \PP_1$ and $C_2 \to \PP_2$, respectively and 
they intersect at $j-2c$ points.  The curves $C_{1\rho}$
(resp.\ $C'_{\nu 1}$)  are produced by taking pairs of points 
from  the components $C_1$ and  $R_{\rho}$
(resp.\ $R_{\nu}$ and $C_1$). 
The curves $C'_{2\mu}$  (resp.\ $C_{\lambda 2}$)  are produced 
by taking pairs of points from  the components $C_2$ and  
$S_{\mu}$  (resp. $S_{\lambda}$ and $C_2$). The curves $R_{\nu \rho}$
(resp.\ $S_{\lambda \mu}$) are produced by taking pairs of points 
from  the components $R_{\nu}$ and $R_{\rho}$
(resp.\ $S_{\lambda}$ and $S_{\mu}$).
\item The curves $S_{\lambda \mu}$, $R_{\nu \rho}$ are rational curves. 
The curves $C_{1 \rho}$ and $C'_{\nu 1}$
are isomorphic to $C_1$, the curves $C_{\lambda 2}$ and $C'_{2\mu}$ are 
isomorphic to $C_2$.
\item The maps $C_{1 \rho } \to C_1$, $C'_{2\mu} \to C_2$, 
$S_{\lambda \mu} \to S_{\lambda}$ and $R_{\nu \rho} \to R_{\nu}$  
are isomorphisms, the map $C_{\lambda 2} \to S_{\lambda}$ is $j+1-c:1$ 
and the map $C'_{\nu 1} \to R_{\nu}$ is $k+1-c:1$.   
The vertical chain of $(-2)$ curves of length $j-2c$  which ends to
$S_{\lambda \mu}$ (resp. $R_{\nu \rho}$) intersects 
$\tilde{T}_{C_1}$ (resp. $\tilde{T}_{C_2}$) at the
point $(p_{\lambda},p_{\mu})$ (resp.\  $(q_{\nu}, q_{\rho})$).  
The chain of $(-2)$  curves which ends to
$S_{\lambda \mu}$ (resp.\ $R_{\nu \rho}$) contracts to  
the point $p_{\lambda}$ (resp. $q_{\nu}$).
\end{enumerate}

\begin{example}
If $k=1$ then the trace curve $T$ for an admissible cover $C \to P$
representing a point of $H_{g,d}$ or a generic point of one of the 
divisors $E_0$, $E_2$, $E_3$ or $E_{j,c}$ equals the curve $C$ and 
the reduced trace curve equals the curve $P$.
For $k=2$ we get as reduced trace curve the curve $C$.
\end{example}

The reduced trace curve is constructed as the quotient of the 
trace curve by the action of the involution. This involution
extends to the smooth model $\tau : \tT \to B$ constructed
in the preceding section. Since the action is fixed point free
outside the diagonal we need to consider this action only 
at the points of the diagonal. 

In case 1, pairs  of smooth points $(a,b)$ 
with $a=b$,  the trace curve has equations $x=\sigma, \; x'=\sigma, \; y+y'=0$. The involution acts by interchanging 
$x$ with $x'$ and $y$ with $y'$. By taking invariant coordinates 
we observe that the quotient is smooth at this point. 
In case 2, pairs of singular points, the only type which involves 
points on the diagonal is type 2.
At these points the involution acts by interchanging $x$ with $x'$, 
$y$ with $y'$ and $u$ with $v$.
When $m$ is even we have a fixed point $[u,v]=[1,-1]$.
The local equations at this point are $[u,v]=[1,-1]$, $xy=\sigma$, 
$x+x'=0$, $y+y'=0$.
By taking invariant coordinates we observe that the quotient has  
an $A_1$ singularity.
There are $[m/2]$ branches on the reduced trace curve but when $m$ is 
even we have to resolve the $A_1$ singularity in the middle by 
inserting a $(-2)$ curve. 
\end{section}
\begin{section}{Generic Finiteness of the Trace Curve Map}
We now prove that the rational map $\phi: \bar{H}_{g,d} \dashrightarrow 
\bM_{g'}$ is generically
finite.

\begin{proposition}\label{genfiniteness}
Let $C$ be a general smooth
 curve of genus $g\geq 4$ and $\gamma$
a base point free $g_d^1$ with $g>2d-4$.
Then the trace curve $T_{\gamma}$ determines $C$ uniquely: if $C'$
is another curve with trace curve $T'$ with $T'$ isomorphic to $T$
then $C'$ is isomorphic to $C$.
\end{proposition}
\begin{proof}
Suppose that $C^{\prime}$ is another smooth curve of genus $g$ with
a pencil $\gamma^{\prime}$ such that $T_{\gamma}$ and $T_{\gamma^{\prime}}$
are isomorphic, say $\psi: T_{\gamma} {\buildrel \sim \over \to}
T_{\gamma^{\prime}}$. Let $p_1$ (resp.\ $p_1^{\prime}$) denote the
first projection of $T$ (resp.\ $T_{\gamma^{\prime}})$.
Then $p_{1,*}\, \psi^{-1} \, (p_1^{\prime})^{*}$
defines a homomorphism $j:{\rm Jac}(C^{\prime}) \to {\rm Jac}(C)$.
We claim that $j$ restricted to a suitable translate of $C^{\prime}$ is
birational to its image.
Since $C$ is general its Jacobian is simple (see e.g.\ \cite{Koizumi,Mori}),
hence $j$ is either zero or an isogeny. If $j$ is zero this means that
for general points $x$ and $y$ in $C^{\prime}$ the divisor
$p_{1,*}\psi^{-1}((p_1^{\prime})^*x)$ is linearly equivalent to
$p_{1,*}\psi^{-1}((p_1^{\prime})^*y)$ and this gives then a pencil
of degree $d-1$ on $C$; since by assumption the Brill-Noether number
$g - (r+1)(g-(d-1)+r)=2d-g-4$ is negative this does not exist on $C$.
Thus $j$ is an isogeny and for a suitable translate of $C^{\prime}$
the map $j$ will be birational.
This image is then a curve of geometric genus $g$ in
${\rm Jac}(C)$ and by a theorem of Bardelli and Pirola for a generic
Jacobian of genus $g \geq 4$ all curves of genus $g$ lying on it are
birationally equivalent to $C$, see \cite{B-P}.

\end{proof}
\begin{example}
Let $C$ be a generic curve of genus $4$. It has two $g^1_3$'s, say $\gamma_1$
and $\gamma_2$. Then the reduced trace curve $\hat{T}$ is isomorphic to
$C$ via the map $r \mapsto p+q$ if $p+q+r \sim \gamma$. But the
trace curves $T_{\gamma_1}$ and $T_{\gamma_2}$ (of genus $13$)
are in general not isomorphic since the maps $T_{\gamma_1} \to C$
and $T_{\gamma_2} \to C$ are branched at different points.
So the map $\phi: H_{4,3} \to M_{13}$ is of degree $(12)!$,
while $\hat{\phi} : H_{4,3} \to M_4$ coincides with the natural
map $p$.
\end{example}

\end{section}
\begin{section}{Intersection Theory on $\overline{M}_{0,b}$}
We recall some basic facts about the divisor theory of
the moduli space of $b$-pointed genus $0$ curves $\bM_{0,b}$, 
see \cite{Ke} (also, \cite{GK2}, section 2).
The boundary of $\bM_{0,b}$ is
the union of  irreducible divisors,
each of which corresponds to a decomposition of $B=\{1,\ldots,b\}$
as $B=\Lambda \sqcup \Lambda^c$ into two disjoint
subsets with $2\leq\# \Lg  \leq b-2$.
We write the corresponding divisor as $S_b^{\Lg}$
modulo the relation $S_b^{\Lg}=S_b^{{\Lg}^c}$.
We sometimes normalize the $\Lambda$ by requiring that
$$
\# (\Lambda \cap \{1,2,3\}) \leq 1.
$$
The map $\varpi: \bM_{0,b+1} \to \bM_{0,b}$ is
equipped with $b$ sections $s_j: \bM_{0,b} \to \bM_{0,b+1}$
with $j=1,\ldots,b$.

The boundary divisors of $\bM_{0,b+1}$ are related to those of
$\bM_{0,b}$ as follows:
$$
\varpi^*S_b^{\Lg}= S_{b+1}^{\Lg} \cup S_{b+1}^{\Lg \cup \{ b+1 \} },
$$
with $\Lg \subset \{1, \ldots, b \}$.
Note that if $\Lambda \subset \{1,\ldots,b\}$ is normalized,
then so are $\Lambda$ and $\Lambda\cup \{ b+1 \}$ as subsets of
$\{1,\ldots,b+1\}$.
So all the boundary components of
$\bM_{0,b+1}$ are coming from $\bM_{0,b}$ except the components
$S_{b+1}^{ \{j, b+1\}}$ ($j=1, \ldots, b$) that correspond to the
image of the $b$ sections $s_j$.

With  $\Lg \subset \{1,\ldots, b\}$, the generic element of
the divisor $S_b^\Lg$ represents a stable curve with two rational
components. Therefore the map $S_{b+1}^{\Lg } \to S_b^{\Lg}$
(resp.\ $S_{b+1}^{\Lg \cup\{b+1\}} \to S_b^\Lg$) is generically a
$\PP ^1$-fibration. We have
\begin{equation*}
S_b^{\Lambda_1} \cap S_{b}^{\Lambda_2} \neq \emptyset
\iff
\# (\Lambda_1 \cup \Lambda_2) \in
\{ \# \Lambda_1, \# \Lambda_2, \# \Lambda_1+\# \Lambda_2, b\}.
\end{equation*}

\begin{definition}
\label{def: symmetricdivisors}
With  $b=6k$ we define on $\bM_{0,b}$ for $2 \leq j \leq 3k-1$ the divisors
$$
\tb ^j=\sum_{\Lambda \subset B, \, \# \Lambda =j} S_b^{\Lambda}
\quad
\mbox{ and } \quad
\tb ^{3k}=\frac{1}{2}
\sum_{\Lambda \subset B, \, \# \Lambda=3k} S_b^{\Lambda} .
$$
\end{definition}
One easily determines the image of $\tH$ under the
morphism $q: \tH \to \bM_{0,6k}$.
\begin{lemma}
The image of $\tilde{H}$ under $q$  is contained in
$$
M_{0,b}\cup \tb ^2\cup \cup_{j=1}^k \tb^{3j} \, .
$$
\end{lemma}
Recall that the $b$ sections $s_i$ define tautological classes $\psi_i$.
\begin{definition}\label{defpsi}
We define a divisor class on $\bM_{0,b}$ by
$$
\psi:= \sum_{i=1}^b \psi_i =
\sum_{j=2}^{b/2} \frac{(b-j)j}{b-1} \, \tb ^j \, .
$$
\end{definition}

\end{section}
\begin{section}{Applying Grothendieck-Riemann-Roch}
In this section we shall apply the Grothendieck-Riemann-Roch Theorem
to the family of trace curves over our $1$-dimensional base $B$
and the relative dualizing sheaf.
We have the diagram 

\begin{equation} \label{diag: T}
\begin{xy}
\xymatrix{
\Tcal'  \ar[dr]^{\eta'} \ar[ddr]_{\tau'} & &  {\tT} \ar[ll]_{\mu} 
\ar[ddl]^{\tau}    \ar[dl]_{\eta}  \ar[d]^{\theta}    & \\
         & \C \ar[d]_{\pi} &   {\Tcal}_{\sigma} \ar[r]^{\phi'} 
\ar[dl]^{\tau_{\sigma}} &   \bar{\mathcal C}_{g^{\prime}} \ar[dl]^{\pi'}  \\
     & B \ar[r]_{\phi}    &   \bM_{g^{\prime}} &   \\
}
\end{xy}
\end{equation}
the notation of which we now explain. The curve ${\Tcal}'/B$ is the 
singular trace curve in which we have performed the small blow-ups 
at the pairs of points of type $2$ and type $3$. It is a nodal 
family of curves.
The curve ${\tT}$ is the smooth model of ${\Tcal}'$ and 
$\theta: {\tT}\to {\Tcal}_{\sigma}$ is the stabilization map. The
space ${\Tcal}'$ has singularities of type $A_m$ 
and the cover  $\eta' : \Tcal' \to \C$ is a finite cover of degree $k$. 
The space ${\tT}$ contains chains of $(-2)$-curves which are obtained
by resolving the singularities of  ${\Tcal}'$.
The map $\eta : {\tT} \to \C$ is a generically finite 
cover of degree $k$.

We wish to calculate $\phi^* \lambda_{\pi'}$, where  $\lambda_{\pi'}$
is the Hodge class of $\bar{\mathcal C}_{g^{\prime}}$ over $\bM_{g'}$.
Note that $(\phi')^* \lambda_{\pi'}=\lambda_{\tau_{\sigma}}$ and since
$\theta: {\tT} \to {\Tcal}_{\sigma}$ is a contraction we have 
$\lambda_{\tau_{\sigma}}=\lambda_{\tau}$, cf.\ Lemma 3.2 in \cite{GK2},
so $\phi^*\lambda_{\pi^{\prime}}= \lambda_{\tau}$.

Application of Grothendieck-Riemann-Roch to $\tau: {\tT} \to B$
gives
$$
12 \lambda_{\tau}= \tau_*(\omega_{\tau}^2)+\delta_{\tau},
$$
where $\delta_{\tau}$ is the push forward of the singularity locus 
of the fibers and $\omega_{\tau}$ denotes the relative dualizing sheaf of
$\tau$, cf.\ \cite{Mumford2} . In order to carry this out we need to calculate
$\tau_*(\omega_{\tau}^2)$ and $\delta_{\tau}$. We begin with the latter.

\begin{proposition}\label{pro: deltatau}
For $k \geq 3$ we have
$$
\delta_{\tau}= (k^2+k) \, E_0+ (2k^2-10k+18) \,E_2 +
(3k^2-13k+16) \, E_3 + \sum_{j,c} d_{j,c} E_{j,c}
$$
with
$$
d_{j,c} =  
[{c \choose 2} +{k-j+c \choose 2}](j+1-2c) 
             +2(c+1) (k-j+c) +j  \, .
$$
Moreover 
$$
\delta_{\tau}=
\begin{cases} 2E_0+E_{1,0} & k=1 \\
6E_0+2E_3+3E_{1,0}+2E_{2,0}+6E_{2,1} & k=2 . \\
\end{cases}
$$
\end{proposition}
\begin{proof}
This formula is obtained by looking at the pictures in Section \ref{Geometry}.
For example, the contribution of $E_2$ consists of a contribution
$2(k-2)(k-3)$ of the $R_{ij}$, a contribution $2(k-3)$ 
of the $S_{1i}$ and $S_{2i}$, a contribution $4$
of the $S'_{\nu \mu}$, a contribution $2$ of the $T_{S_{\nu}}$, 
  a contribution of $2(k-3)$ of the $S^{\prime\prime}_{\nu j}$, giving in
total $2(k-3)(k-4) + 4(k-3)+6= 2k^2-10k+18$.
The other coefficients are obtained in a similar way. For example,
for the case of $E_{j,c}$ we find a contribution $2\, {k-j+c \choose 2}$
from the chains ending with $S_{\lambda  \mu}$; similarly $2\, {c \choose 2}$
from those ending with $R_{\nu \rho}$, 
a contribution $2\, c(k-j+c)$ from the intersections
$C_{1\rho} \cdot C_{\lambda 2}$ and $C_{\nu 1} \cdot C_{2\mu}$,
a contribution 
$2c+2(k-j+c)$ from the intersections $\tilde{T}_{C_2}\cdot C_{\nu 1}$,
$\tilde{T}_{C_2}\cdot C_{\lambda 2}$, $\tilde{T}_{C_1}\cdot C_{2\mu}$,
$\tilde{T}_{C_1}\cdot C_{1\rho}$ and finally 
$j-2c$ from the intersections
of $\tilde{T}_{C_1}$ with $\tilde{T}_{C_2}$.
\end{proof}

\begin{remark}
Note that the formula for $k\geq 3$ remains valid if we interpret $E_2$
and $E_3$ (resp.\ $E_2$) as zero for $k=1$ (resp.\ for $k=2$).
\end{remark}

Now we turn to the calculation of $\tau_*(\omega_{\tau}^2)$. A first remark is 
that (in additive notation)
$$
\omega _{\tau} = \eta^*\omega_{\pi}+R_{\eta},
$$
with $R_{\eta}=\mu^*R_{\eta'}$, where $R_{\eta'}$  is the ramification locus 
of the finite map $\eta': \Tcal ' \to \C$.
This is the same as the closure of the  ramification locus of the map 
$\tau'$ (or $\tau $) restricted to the locus of $B$
which represents smooth curves.  Note that $R_{\eta'}$ is supported  
outside of the singular locus of $\Tcal'$
and so it defines a Cartier divisor on $\Tcal'$. 
The formula above is derived by applying $\mu^*$ to the formula
$\omega_{\tau'} = ({\eta'})^*\omega_{\pi}+R_{\eta'}$; 
the latter holds  because it holds outside of the singularities
of the spaces $\Tcal$ and $\C$. Since 
$\mu^*\omega _{\tau'} = \omega_{\tau}$ the formula follows.  

We calculate
$$
\omega_{\tau}^2 = \eta^*\omega_{\pi}^2 +2 \, \eta^*\omega_{\pi} 
\cdot R_{\eta } + R_{\eta}^2
$$
and observe
$$
\tau_*( \eta^*\omega_{\pi}^2 )  = \pi_* \eta_*(\eta^*\omega_{\pi}^2)
                              = k \; \pi_*(\omega_{\pi}^2)\, ,
$$
because $\eta $ is a generically finite map of degree $k$.

Note that $\C$ is a singular space but all the above cycles 
represent Cartier divisors, so the intersection product makes sense.
In the calculation we use the following diagram (\ref{diag: C}) with 
$\tilde{\C}$ the smooth model of $\C$ and with $b=6k$ in
 $\bM_{0,b}$ and $\bM_{0,b+1}$.

\begin{equation} \label{diag: C}
\begin{xy}
\xymatrix{
 &     &  \tilde{\C}  \ar[dll]_{r} \ar[ddl]^{\tilde{\pi}} \ar[dl]_{\nu}    \\
 \bM_{0,b+1} \ar[d]_{\varpi}  & \C \ar[d]_{\pi\;\;} \ar[l]^{\gamma } &   \\
 \bM_{0,b}   &   B \ar[l]^{q}   &   \\
}
\end{xy}
\end{equation}
If $c$ is a cycle on $\C$ we have $\pi_*c = \tilde{\pi}_*\nu^* c$ 
because $\nu_*\nu^*c =c$. 
Since now
$\nu^* \omega_{\pi} = \omega_{\tilde{\pi}}$ 
we get
\begin{equation} \label{eq: omegat2}
\tau_*(\eta^*\omega_{\pi}^2 ) = k\; \tilde{\pi}_*\omega_{\tilde{\pi}}^2 \, .
\end{equation}
We also have
$$
\tau_* (\eta^*\omega_{\pi} \cdot R_{\eta } )  = 
\pi_* \eta_* (\eta^*\omega_{\pi} \cdot R_{\eta })
=\pi_* (\omega_{\pi} \cdot \eta_* R_{\eta})\, .
$$
The trace curve is ramified over $\C$ in the points $(p,q)$ in the fibre
over $p \in \C$ where $q$ is a ramification point of the map $\gamma$. 
This implies
$$
\nu^* \eta_* R_{\eta} = r^* \hat{S} - 2R_r,
$$
with $\hat{S}= \sum_{i=1}^b S_{b+1}^{\{i,b+1\}}$ the sum of the 
image of the sections of the map $\varpi$ and $R_r$ is the closure of 
the ramification of $r$ over the smooth locus. 
This yields

\begin{equation} \label{eq: omegatR}
\tau_* (\eta^*\omega_{\pi} \cdot R_{\eta } ) = 
\tilde{\pi}_*[\omega_{\tilde{\pi}} \cdot (r^* \hat{S} - 2R_r )] \, .
\end{equation}
The right hand sides of (\ref{eq: omegat2}) and (\ref{eq: omegatR}) 
can be calculated  in a way similar to the calculations in
our paper \cite{GK2}.
In order to calculate $\tau_*(R_{\eta}^2)$ we will use that the 
map $\eta'$  (and $\eta $) is a simple cover and therefore if 
$V=\eta_*R_{\eta}$ is the branch locus of $\eta '$ (or $\eta$) then
$\eta^* V = 2 R_{\eta } + R'_{\eta}$ with $R_{\eta } \cdot R'_{\eta }=0$.  
Therefore,

\begin{equation} \label{eq: R2}
\begin{aligned}
\tau_* R_{\eta}^2 & = \frac{1}{2} \; \tau_*( R_{\eta}\cdot\eta^*V) = 
\frac{1}{2} \; \tilde{\pi}_*(\eta_*R_{\eta} \cdot V)\\
        &  =\frac{1}{2} \; \tilde{\pi}_*(V^2) = \frac{1}{2} \; 
\tilde{\pi}_*(\nu^*V^2) = \frac{1}{2}\; 
\tilde{\pi}_*[(r^* \hat{S} - 2R_r )^2]\, .
\end{aligned}
\end{equation}

Since we are dealing with the divisors $E_0$, $E_2$, $E_3$ and $E_{j,c}$ only
we may adapt the earlier definition of the  divisor class $\psi$ 
on $\bM_{0,b}$ by setting
\begin{equation}  \label{eq: psi}
\psi:= \sum_{i=1}^b \psi_i =
\frac{2(b-2)}{b-1} \, \tb^2 +   
\sum_{j=1}^{k} \frac{3j(b-3j)}{b-1} \, \tb ^{3j} \, .
\end{equation}

The following formulas are a consequence of Lemma 3.1 in our \cite{GK2}:
\begin{equation}  \label{eq:pullbackT}
\begin{aligned}
q^*T_b^2 & = E_0 + 2E_2 +3E_3 \, ,\\
q^*T_b^{3j} & = \sum _{c=0}^{[j/2]} (j+1-2c) E_{j,c}, \; j=1, \ldots, k \, .
\end{aligned}
\end{equation}

Carrying out the calculations as in \cite{GK2} for the right hand sides
of the equations (\ref{eq: omegat2}),(\ref{eq: omegatR})  and
(\ref{eq: R2}) the  following formulas can be deduced from \cite{GK2}
(cf.\ Lemma 4.2 there):

\begin{lemma} \label{prop: GK2}
We have the following identities
$$
\begin{aligned}
\tilde{\pi}_*(r^* \omega_{\varpi} \cdot R_r) & = q^* \psi, \quad
\tilde{\pi}_*(\omega_{\tilde{\pi}} \cdot R_r)  =  \frac{1}{2}\;  
q^* \psi, \quad 
\tilde{\pi}_*(R_r^{2})   = - \frac{1}{2} \;  q^* \psi  , \\
\tilde{\pi}_*(\omega_{\tilde{\pi}}^{2}) &= (3/2) q^* \psi - (k+1) \; 
q^*( T_b^2  +\sum_{j=1}^{k}  \tb ^{3j}) \, .
\end{aligned}
$$
\end{lemma}

As a check please note that for $k=1$ the formula for 
$\tilde{\pi}_*(\omega_{\tilde{\pi}}^{2})$ gives 
$5\kappa_1=E_0+7E_{1,0}$, in agreement with $5\kappa_1=\delta_0+7\delta_1$ 
(see \cite{Mumford}, Eqn (8.5)).

We will need the following lemma.
\begin{lemma} \label{Gamma}
If $\Gamma$ is a cycle on $\bM_{0,b+1}$ then
$ \tilde{\pi}_*r^*\, \Gamma= (k+1)\, q^* \varpi_{*} \,\Gamma$.
\end{lemma}
\begin{proof}
Let $P=\bM_{0,b+1} \times_{\bM_{0,b}} B$ be the fibre product.
The induced map $\xi: \tilde{\C} \to P$ is  generically a
$(k+1):1$ map. Let $r_1: P \to \bM_{0,b+1}$ and $r_2: P \to B$
be the projections. Then $\xi_*\xi^* =k+1$.
We have $\tilde{\pi}_*r^*\, \Gamma= \tilde{\pi}_* \xi^* r_1^*\, \Gamma=
r_{2*} \xi_* \xi^* r_1^*\, \Gamma
= (k+1) \, r_{2*}r_1^*\, \Gamma= (k+1) \,q^* \varpi_{*} \, \Gamma$.
\end{proof}

\begin{lemma}\label{auxiliarypush}
We have
$$
\tilde{\pi}_*(r^* \hat{S}^2)= -(k+1)q^*\psi, \quad
\tilde{\pi}_*(R_r \cdot r^* \hat{S})= -q^* \psi, \quad
\tilde{\pi}_*(\omega_{\tilde{\pi}}\cdot r^* \hat{S})= k \, q^* \psi.
$$
\end{lemma}
\begin{proof}
By the adjunction formula we have $\varpi_*(\hat{S}^2)= -\psi$. Therefore,
$\tilde{\pi}_*(r^*\hat{S}^2)=(k+1)q^* \varpi_*(\hat{S}^2)=-(k+1) q^*\psi$.
For the second formula, if we denote the ramification sections
of $\tilde{\pi}$ by $\rho_i: B \to \tilde{\C}$ then $r \circ \rho_i=
s_i \circ q$ with $s_i$ the sections of $\varpi$.
We have 
$\tilde{\pi}_*(r^*\hat{S} \cdot R_r) = 
\sum_i \tilde{\pi}_*(\rho_{i*} \rho_i^* r^*\hat{S}) =  
\sum_i \tilde{\pi}_*\rho_{i*}( q^* s_i^*\hat{S}) 
= - \sum _i q^* \psi_i = - q^* \psi $.
For the third, we have $\omega_{\tilde{\pi}}=r^*\omega_{\varpi}-R_r$;
so 
$\tilde{\pi}_*( \omega_{\tilde{\pi}} \cdot r^*\hat{S}) 
=\tilde{\pi}_*((r^*\omega_{\varpi}-R_r)  \cdot r^*\hat{S}) $
and by the second formula
this equals
$\tilde{\pi}_*((r^*\omega_{\varpi}\cdot r^*\hat{S})- q^*\psi$.
It thus suffices to show 
$$
\tilde{\pi}_*(r^* \omega_{\varpi} \cdot r^* \hat{S}) = (k+1) \; q^* \psi \, .
$$
But we have
$$
\varpi_*(\omega_{\varpi} \cdot \hat{S})= \varpi_*(\sum_i s_{i *} s_i^* \omega_{\varpi})=
\sum_i \varpi_* s_{i*} \psi_i=\sum_i \psi_i=\psi\, .
$$
Now apply lemma \ref{Gamma}.
\end{proof}

As a corollary of equations (\ref{eq: R2}) and (\ref{eq: omegatR}), 
and Lemmas \ref{prop: GK2} and \ref{auxiliarypush}
we get the following formulas.

\begin{corollary}
For  $k\geq 1$ we have
$$
\tau_*(R_{\eta}^2) = -{1/2} \, (k-1) \, q^* \psi\, , \qquad 
\tau_* (\eta^*\omega_{\pi} \cdot R_{\eta})  = (k-1)\, q^* \psi
$$ and
$$ 
\tau_*( \eta^*\omega_{\pi}^2 )  =   \frac{3k}{2} q^* \psi
 - k(k+1) \; q^*( T_b^2  +\sum_{j=1}^{k}  \tb ^{3j}) \, .                       
$$
\end{corollary}

Substituting these formulas in $\tau_*(\omega_{\tau}^2)$ we find
\begin{proposition}\label{pro: omegatau2}
For $k\geq 1$ we have
$$
\tau_{*}(\omega_{\tau}^2)=\frac{-6k^3+31k^2-29k+6}{6k-1} (E_0 + 2E_2+3E_3)
+ \sum_{j=1}^k \sum_{c=0}^{[j/2]} a_{j,c} E_{j,c},
$$
with $a_{j,c}$ given by
$$
(j+1-2c)\left( \frac{27}{2} \frac{j (2k-1)(2k-j)}{6k-1} -k(k+1) \right) \, .
$$
\end{proposition}

By substituting the formulas of propositions 
\ref{pro: deltatau} and \ref{pro: omegatau2} in 
the expression of $12\, \phi^* \lambda_{\pi'}$ given 
by the Grothendieck-Riemann-Roch theorem we have 
\begin{theorem} \label{th: lambdag'}
The pull back of the Hodge class $\lambda_{g'}$ 
of $\bM_{g'}$ under $\phi$ equals
$$
12\, \phi^* \lambda_{g'}=\frac{2}{6k-1}(t_0\, E_0+t_2\, E_2+t_3\, E_3)+ 
\sum_{j=1}^k \sum_{c=0}^{[j/2]} t_{j,c} \, E_{j,c}
$$
with the coefficients $t_0,t_2,t_3$ and $t_{j,c}$ defined by
$$
t_0=18\, k^2-15\, k+3, \quad t_2=30\, k-3, \quad t_3=6\, k^2+11\, k+1
$$
and $t_{j,c}=a_{j,c}+d_{j,c}$.
\end{theorem}
\begin{example}
Take $k=1$ and interpret $E_2$ and $E_3$ as zero. 
Since $\phi$ is the map $p:\tH \to \bM_2$ 
we get the formula for the  Hodge bundle on $\tilde{H}_{2,2}$; it says
$\lambda_{\tilde{H}_{2,2}}=(E_0 + E_{1,0})/5$.
This fits with the formula given in Thm.\ 1.1 of \cite{GK2}, cf.\ also
Prop.\ 8.1 of \cite{GK}.
\end{example}
\end{section}

\begin{section}{The Reduced Trace Curve}
We carry out the analogous calculations for the reduced trace curve
and calculate the pull back of the Hodge class on $\bM_{\hat{g}}$
under $\hat{\phi}$.

We denote the family of the reduced trace curves over our $1$-dimensional
base $B$ by ${\S}^{\prime}$ and the smooth model 
(obtained by resolving the $A_1$ singularities
coming from the isolated fixed points) by ${\S}$. We have the quotient map
$\sigma: {\tT} \to {\S}^{\prime}$.
Note that $\omega_{s}$ is trivial in a neighborhood of an $A_1$ resolution
and the pullback of $\omega_{s'}$ to ${\S}$ is $\omega_{s}$.
We have the diagram

\begin{equation} \label{diag: reducedtrace}
\begin{xy}
\xymatrix{
 && \tT \ar[ddl]^{\tau } \ar[dl]_{\eta} \ar[dr]^{\sigma}  &  \\
 \bM_{0,b+1} \ar[d]_{\varpi}  & {\C} \ar[d]_{\pi} \ar[l]_{\quad \gamma} 
&  & {\S}^{\prime} \ar[dll]_{s'} & {\S}\ar[l] \ar[dlll]^{s}  \\   
 \bM_{0,b}   &   B \ar[l]^{q}   &  & \\
}
\end{xy}
\end{equation}

\begin{lemma}\label{sstaromegas2}
We have $s_*(\omega_s^2)= \frac{1}{2} \tau_*(\omega_{\tau}^2)
- \frac{3}{4} q^*(\psi)$.
\end{lemma}
\begin{proof} Since the singularities of ${\S}^{\prime}$ are of type $A_1$
we can and shall neglect them for this calculation and work on ${\S}'$.
We have $\omega_{\tau}=\sigma^*\omega_{s'}+R_{\sigma}$ with 
$R_{\sigma}$ the ramification divisor of $\sigma$, hence 
$$
\omega_s^2 = \frac{1}{2}\sigma_*\sigma^* (\omega_{s'}^2) =
\frac{1}{2}\sigma_*[(\sigma^*\omega_{s'})^2]=
\frac{1}{2}\sigma_*[(\omega_{\tau} - R_{\sigma})^2]
$$
and thus
$$
s_*(\omega_s^2)  =\frac{1}{2}\tau_*[(\omega_{\tau} - R_{\sigma})^2] 
=\frac{1}{2} \tau_*(\omega_{\tau}^2)- \tau_*(\omega_{\tau} \cdot R_{\sigma}) + 
\frac{1}{2} \tau_*(R_{\sigma}^2) .
$$
We denote by $R_{\gamma}$  the closure of 
the ramification of $\gamma$ over the smooth locus. Note that $\nu^* R_{\gamma }=R_r$.
We have $R_{\sigma} \cdot R_{\eta} = 0$ 
because $\eta (R_{\sigma})=R_{\gamma}$ and $\eta(R_{\eta}) = \gamma^{-1}(\hat{S})-R_{\gamma}$. 
We have 
$$
\tau_*(\omega_{\tau} \cdot R_{\sigma})  
= \tau_*[(\eta^*\omega_{\pi} +R_{\eta}) \cdot R_{\sigma}]
= \tau_*[\eta^*\omega_{\pi} \cdot R_{\sigma}] = 
\pi_*(\omega_{\pi} \cdot R_{\gamma})= \tilde{\pi}_*(\omega_{\tilde{\pi}} \cdot R_r)
 =\frac{1}{2} q^* \psi \, , 
$$
by Lemma \ref{prop: GK2}. 
If $\iota$ denotes the involution on $\tT$ then $\eta^* R_{\gamma}= R_{\sigma}+
\iota^* R_{\eta}$ and $R_{\sigma} \cdot \iota^*R_{\eta}=0$.
Furthermore we have
$$
\tau_*(R_{\sigma}^2)  = \tau_*(R_{\sigma} \cdot \eta^*R_{\gamma}) = \pi_*[(R_{\gamma})^2] = \tilde{\pi}_* [(R_r)^2]
 = -\frac{1}{2} q^* \psi \, . 
$$
\end{proof}

\begin{lemma}
The push forward $\delta_s$ of the locus of singularities of the fibers of $s$
for $k\geq 3$ is given by
$$
\delta_s= \frac{k^2+k}{2}\, E_0 + (k^2-5\, k +12) \, E_2+
\frac{3\, k^2-13 \, k +16}{2} \, E_3 + \sum_{j=1}^k \sum_{c=0}^{[j/2]}
s_{j,c} \, E_{j,c},
$$
with $s_{j,c}$ given by
$$
\begin{aligned}
s_{j,c}=
(k-j+c)(c+1)+\left( {k-j+c \choose 2} + {c \choose 2}\right) (j+1-2c) + 
& \\
[\frac{j+1}{2}]+
\begin{cases} 1 & \hbox{$j$ odd} \\ 0 & else.   \end{cases} \\
\end{aligned}
$$
and for $k=1$ and $k=2$ by
$$
\delta_s = \begin{cases}
E_0+E_{1,0} & k=1 \\
3\, E_0+E_3+3\, E_{1,0}+E_{2,0}+3\, E_{2,1} & k=2 \\
\end{cases}
$$
\end{lemma}
\begin{proof}
We use the local description of the reduced trace curve given in 
section \ref{Geometry}.
The contribution of $E_0$ to $\delta_s$
is ${k-1 \choose 2} + 2(k-1)= (k^2+k)/2$. From $E_2$ we find 
the contribution $2 {k-2 \choose 2} + 2(k-3) +2+4=k^2-5\, k +10$,
where the last $4$ comes from the points $(p_{\nu},p_{\nu})$ for
$\nu=1,2$ that give an $A_1$-singularity on the reduced trace curve.
For $E_{j,c}$ note that for $j$ odd 
the `middle' intersection point of $\tilde{T}_{C_1}$
and $\tilde{T}_{C_2}$ gives rise to an $A_1$-singularity on the
trace curve. The other contributions are obtained similarly.
\end{proof}

\begin{remark}
By interpreting for $k=1$ (resp.\ $k=2$) 
the divisors $E_2$ and $E_3$ (resp.\ $E_2$)
as zero, the formula for $k\geq 3$ works for all $k\geq 1$.
As a check on the formula note that for $k=1$ the reduced trace
curve of $C \to P$ equals $P$ and we thus easily see
that we have coefficients $1$ for both $E_0$ and $E_{1,0}$.
Similarly, for $k=2$ the reduced trace curve equals $C$
and we thus can easily read off from the left hand side
of the Figures 1-4 the multiplicities.
\end{remark}
By substituting the formulas for $\tau_*(\omega_{\tau}^2)$ and $q^*\psi$
in $s_*(\omega_s^2)$ in Lemma \ref{sstaromegas2} and adding $\delta_s$ we get 
an expression for $12 \, \hat{\phi}^* \lambda_{\hat{g}}$.

\begin{theorem}\label{th: lambdahatg}
For $k\geq 3$ the pull back of the Hodge class $\lambda_{\hat{g}}$ 
of $\bM_{\hat{g}}$ under $\hat{\phi}$ is given by
$$
12 \,  \hat{\phi}^* \lambda_{\hat{g}}=
\frac{2}{6k-1}\left( u_0 \, E_0 + u_2\, E_2 + u_3\, E_3 \right) + 
\sum_{j=1}^k \sum_{c=0}^{[j/2]} u_{j,c}\, E_{j,c}
$$
with $u_0=9\, k^2-12\, k+3$, $u_2=15\, k$, $u_3=3\, k^2-8\, k +5$ and
$$
u_{j,c}=s_{j,c} -\frac{(j+1-2c)}{2(6k-1)}
\left( (27k-27)j^2-54(k^2-k)\, j + (k^2+k)(6k-1) \right) \, .
$$
\end{theorem}
\begin{example}
Take $k=2$ and interpret $E_2$ as zero. The formula says that
$$
12 \hat{\phi}^* \lambda_{\hat{g}}=\frac{30}{11} \, E_0+\frac{2}{11} \, E_3+
\frac{48}{11} \, E_{1,0}+ \frac{74}{11} \, E_{2,0}+ \frac{54}{11} \, E_{2,1}.
$$
Comparing this with the formula for the Hodge class of $\tilde{H}_{4,3}$
(cf.\ \cite{GK2}) we see that it fits.
\end{example}
\end{section}

\begin{section}{Pulling Back Boundary Divisors}
We shall need to know the pull backs of the boundary divisors 
$\delta_j^{\prime}$
in $\bM_{g'}$ (resp.\ $\hat{\delta}$ in $\bM_{\hat{g}}$) 
under the rational maps
$\phi: \tH \to \bM_{g'}$ (resp.\ $\hat{\phi}: \tH \to \bM_{\hat{g}}$).

\begin{proposition} \label{prop: pullbackdeltaprime}
For $k\geq 3$ the pullback $\phi^*(\delta_0^{\prime})$ equals
$$
(4k-2)\, E_0+4\, E_2+2\, E_3+  
 \sum_{j=2}^k\sum_{c=1}^{[j/2]} 
\left( 2(k-j+c)(c+1)+j \right) \,E_{j,c}  + \\
\sum_{j=2}^k j\, E_{j,0}  \, ;
$$
furthermore, $ \phi^*(\delta_1^{\prime})= (2k-1)E_{1,0}$ 
and
$$
\phi^*(\delta_j^{\prime})= 
\begin{cases}
(2k-2j)\, E_{j,0} & j=2,\ldots,k \\
0  & \mbox{else} \\
\end{cases}
$$
\end{proposition}

\begin{proof}
To prove this formula for the pull back of $\delta^{\prime}_0$
(resp.\ $\delta^{\prime}_j$, $j\geq 1$)
we count in a $1$-dimensional family of
semi-stable models of trace curves ${\tT} \to B$ the number
of non-disconnecting nodes  (resp.\ of disconnecting nodes that split the curve
in a component of genus $j$ and one of genus $g'-j$).
The semi-stable model of the trace curve over a generic point of $E_0$  
has $2(2\, k-1)$ non-disconnecting nodes. The semi-stable model of the 
trace curve over a generic point of $E_2$  has $4$ non-disconnecting nodes.
Over a generic point of $E_3$  has $2$ non-disconnecting nodes and finally,
over a generic point of $E_{j,c}$ the situation is:
for $c \geq 1$ it has $2(c+1)(k-j+c)+j$  non-disconnecting nodes;
for $c=0$ and $j \geq 2$ it has $j$ non-disconnecting nodes and $2(k-j)$ 
disconnecting nodes of type $j$, while for $c=0$ and $j=1$ it has 
$2k-1$ disconnecting nodes of type $1$.
\end{proof}
In a similar way we derive the following proposition.
\begin{proposition} \label{prop: pullbackdeltahat}
For $k\geq 3$ the pull back
$\hat{\phi}^*(\hat{\delta}_0)$
equals 
$$
\begin{aligned}
(2k-2)\, E_0+2\, E_2+
\sum_{j=2}^k\sum_{c=1}^{[j/2]}
\left( (k-j+c)(c+1)+[\frac{j+1}{2}] + \epsilon \right) \,E_{j,c}  + & \\
\sum_{j=3}^k
([\frac{j+1}{2}] +\epsilon)\, E_{j,0} & \, ;\\
\end{aligned}
$$
with $\epsilon =0$ if $j=$even, $\epsilon =1$ if $j=$odd and $\epsilon=-1$ if $j=2,\, c=1$;
furthermore,
$
\hat{\phi}^*(\hat{\delta}_1)= (k-1)\, E_{1,0}$, $\hat{\phi}^*(\hat{\delta}_2)= (k-1)\, E_{2,0}$
and
$\hat{\phi}^*(\hat{\delta}_j)= (k-j)\, E_{j,0}$ for $j=3,\ldots,k$
while $\hat{\phi}^*(\hat{\delta}_j)=0$ for $j>k$.
\end{proposition}
\end{section}
\begin{section}{Push forward to $\bM_{g}$}\label{pushforward}
In \cite{GK} we have calculated the push forwards of the boundary classes
under $p: \tH \to \bM_g$. The result is as follows. Let 
$$
N=N(k)= \frac{1}{k+1} {2k \choose k}
$$
then 
$$
\frac{2}{(6k)!}  \, p_*E_0 = N\,  \delta_0,
$$
and
$$
\begin{aligned}
\frac{2}{(6k)!} p_*E_2 &=  \frac{2(k-2) N}{(2k-1)}[ (18\, k^2+51\, k -9) \lambda -
(3\, k^2+4\, k -1) \delta_0]+ \sum_{j=1}^k c_j \delta_j  \, ,\\
\frac{2}{(6k)!}\, p_*E_3 &=  \frac{ 3N}{(2k-1)}\, [ (12\, k^2+46\, k-8) \, \lambda-  
(2k^2+4k-1)\delta_0] - \sum_{j=1}^k \frac{3N}{2k-1} \, b_j \, \delta_j \, ,\\
\end{aligned}
$$
where the $c_j$ and $b_j$ are given in \cite{GK}, Thm 1.1 and Section 8
and with 
$$
e_{j,c}= 
\frac{ (j+1-2c)^2}{(j+1)(2k-j+1)} 
{j+1 \choose c} {2k-j+1 \choose k+1-c} 
$$
we have finally
$$
\frac{1}{(6k)!} \, p_*E_{j,c} = e_{j,c} \, \delta_j \, .
$$
By substituting the above formulas in the expression of 
$\phi^* \lambda_{g'}$ given in Theorem \ref{th: lambdag'} 
we get the following theorem.

\begin{theorem}\label{th:lambda}
For $k\geq 3$ the push forward 
$\frac{1}{(6k)!} \, p_* \phi^* \lambda_{g'}$ on $\bM_{g}$ equals
$$
\begin{aligned}
&\frac{N (18\, k^3+31\, k^2 -69 \, k +11)}{2k-1} \, \lambda
- \frac{N (3\, k^3-5\, k +1)}{2k-1} \, \delta_0\\
& - \sum_{j=1}^k \left( 
\frac{-(10\, k -1)}{4(6\, k -1)} c_j + 
\frac{N \, (6\, k^2 +11\, k -1)}{4(12\, k^2 - 8\, k +1)} \, b_j+
\frac{1}{12} \sum_{c=0}^{[j/2]} e_{j,c}(a_{j,c}+d_{j,c}) \right) \, \delta_j \,. 
\end{aligned}
$$
\end{theorem}

By theorem \ref{th: lambdahatg} we have a similar theorem 
for the map defined by the reduced trace curve.

\begin{theorem}
For $k\geq 3$ the push forward
$\frac{1}{(6k)!} \, p_* \hat{\phi}^{*}\lambda_{\hat{g}}$ equals
$$
\begin{aligned}
&\frac{N \, (18 k^3+19k^2 -117 k +20)}{2(2k-1)} \, \lambda
- \frac{N\, (k-2)(3k^2+4k-1)}{2(2k -1)} \delta_0  \\
& - \sum_{j=1}^k \left(\frac{(-5k-1)}{4(6k-1)} c_j + \frac{N(3k^2-8k+5)}{4(6k-1)(2k-1)} b_j-
 \frac{1}{12} \sum_{c=0}^{[j/2]} e_{j,c} u_{j,c} \right) \, \delta_j \, . \\
\end{aligned}
$$
\end{theorem}

Proposition \ref{prop: pullbackdeltaprime} yields the following result.
\begin{proposition}\label{prop: pushforwarddeltaprime}
The action induced by the correspondence of the boundary divisors
$\delta_j^{\prime}$ for $j=0,\ldots,[g^{\prime}/2]$ of $\bM_{g^{\prime}}$
is given by:
$$
p_*\phi^* \delta_0^{\prime}= w_{\lambda}\, \lambda + w_0 \, \delta_0 +
\sum w_j \, \delta_j \, ,
$$
where
$$
w_{\lambda}= \frac{6 (6k)! \, N (6\, k- 1)}{2\, k -1} 
(2\, k^2+3\, k -8), \quad w_0=-\frac{2(6k)! N}{2\, k-1}
(6 \, k^3- 3\, k^2 -10 \, k +2)
$$
and $w_1=(6k)!(2\, c_1- 3\, N\, b_1/(2\, k-1))$  and
for $2\leq j \leq k$
$$
w_j=j\, e_{j,0}+\sum_{c=1}^{[j/2]} e_{j,c}\, (2(k-j+c)(c+1)+j)
+2\, (6k)!\, c_j -  (6k)! \frac{3N}{2k-1} \, b_j \, ;
$$
furthermore,
$ p_*\phi^* \delta_1^{\prime}= (2\, k-1) \, e_{1,0} \, \delta_1 $
and $p_*\phi^* \delta_j^{\prime}= (2\, k- 2\, j) \, \delta_j $
for $j=2,\ldots,k$
and $p_*\phi^* \delta_j^{\prime}=0$ for $j>k$.
\end{proposition}

Similarly, Proposition  \ref{prop: pullbackdeltahat} yields the following
result.
\begin{proposition}\label{prop: pushforwarddeltahat}
The action induced by the correspondence of the boundary divisors
$\delta_j^{\prime}$ for $j=0,\ldots,[\hat{g}/2]$ of $\bM_{g^{\prime}}$
is given by:
$$
p_*\phi^* \hat{\delta}_0 = v_{\lambda}\, \lambda + v_0 \, \delta_0 +
\sum v_j \, \delta_j \, ,
$$
where
$$
v_{\lambda}=\frac{6N (6k)!(6k-1)}{2k-1} \; (k+3)(k-2)  
\quad     v_0 = - \frac{N (6k)!}{2k-1} \; (6k^3-6k^2-15k+3) 
$$
and $v_1=(6k)!\, 2\, c_1$  and
for $2\leq j \leq k$
$$
v_j = ([\frac{j+1 }{ 2}]+\epsilon) \, e_{j,0} +\sum_{c=1}^{[j/2]} e_{j,c}\, ( (k-j+c)(c+1)+[\frac{j+1}{2}]+\epsilon )
+2\, (6k)!\, c_j \, ,
$$
with $\epsilon $ as defined in Proposition \ref{prop: pullbackdeltahat};
furthermore,
$ p_*\hat{\phi}^* \hat{\delta}_1= (k-1) \, e_{1,0} \, \delta_1 $,
$\;  p_*\hat{\phi}^* \hat{\delta}_2= (k-1) \, e_{2,0} \, \delta_2 $
and $p_*\hat{\phi}^* \hat{\delta}_j= (k-j) \, \delta_j $
for $j=3,\ldots,k$
and $p_*\hat{\phi}^* \hat{\delta}_j =0$ for $j>k$.
\end{proposition}
\end{section}
\begin{section}{Slopes}
We consider again the correspondence
$$
\bar{M}_{2k} {\buildrel p \over \longleftarrow} \tilde{H} 
{\buildrel \phi \over \longrightarrow} \bM_{g'}
$$
It acts on the Picard group via $D \mapsto p_{*}\phi^* D$. 
We now show that it maps ample divisors 
of $\bM_{g'}$ to moving divisors of $\bM_{g}$. 
A moving divisor is a divisor $D$ such that the base locus
of all the linear systems $|mD|$ with $m\geq 1$ is of codimension at least~$2$.

\begin{lemma}  \label{le: ample-moving}
If $D'$ is an ample divisor on $\bM_{g'}$ then the divisor
$D:=p_*\phi^*D'$  is a moving divisor. In other words, the correspondence 
sends the ample cone of $\bM_{g'}$ to the moving cone of $\bM_g$.
\end{lemma}
\begin{proof}  Let $\tilde{M}_g$ be the locus of $\bM_g$ where the map 
$p: \tH \to \bM_g$ is finite. Since $\bar{H}_{g,d}$ is irreducible
the complement of $\tilde{M}_g$ in $\bM_g$  is of codimension $\geq 2$.
We shall show that the common base locus
${\bf B}(D)=:=\cap_{m\geq 1} {\rm Base}(|mD|)$  
is a subset of the above complement. 
Indeed, if $x \in \tilde{M}_g$ we shall show that $x \notin {\bf B}(D)$. 
Let $p^{-1}(x) = \{ h_1, \ldots, h_{N_0} \}$ and let 
$A=\{  \phi(h_1),\ldots,\phi (h_{N_0}) \}$. As we may assume that
$mD'$ is very ample for appropriate $m$, 
we can choose a divisor $Z$ in  $|mD'|$  with $Z \cap A =\emptyset$.
But then $p_*\phi^*Z$ is an element of $|mD|$ that does not contain $x$.
\end{proof}

We write $\delta_j'$ ($j=0,\ldots,[g'/2]$) for the boundary divisors of 
$\bM_{g'}$ and put $\delta'=\sum_{j=0}^{g'/2}\delta_j^{\prime}$. 
The ample cone of $\bM_{g'}$
is well-known: a divisor $D'= x\, \lambda' - y \, \delta'$ is ample
if and only if $x>11\, y$. Given a divisor $D'= x\, \lambda' - y \, \delta'$
we wish to determine the slope $s$ of the induced divisor $p_*\phi^* D'$
in terms of the slope $s'= x/y$ of $D'$. We write
$$
\begin{aligned}
p_*\phi^* \lambda' &= 
\alpha_{\lambda} \lambda - \alpha_0 \delta_0 -\sum_{j=1}^k \alpha_j \delta_j
\, , \\
p_*\phi^* \delta'_0 &= 
\beta_{\lambda} \lambda + \beta_0 \delta_0 +\sum_{j=1}^k \beta_j \delta_j 
\, , \\
\end{aligned}
$$
and
$$
p_*\phi^* \delta'_j  = \gamma _j \delta_j \quad {\rm for} \quad 
j=1, \ldots, k
$$
while $p_*\phi^* \delta'_{\nu}=0$
for $\nu > k$
with the coefficients determined in 
Theorems \ref{th:lambda} and~\ref{prop: pushforwarddeltaprime}.
The divisor  $p_*\phi^*D'$ can be written as
$$
 p_*\phi^*D' =  (x\, \alpha_{\lambda} - y\,  \beta_{\lambda}) \, \lambda 
- (x\, \alpha_0 +y\,  \beta_0)\,\delta_0 - \sum_{j=1}^k 
(x\alpha_j+y \beta_j+y\gamma_j)\, \delta_j \, .
$$
Thus the slope is given by
$$
{\rm slope}(p_*\phi^*D') =  
\frac{x \, \alpha_{\lambda} - y\, \beta_{\lambda}}{x\, \alpha_0 +y\, \beta_0} 
=\frac{\alpha_{\lambda} s'- \beta_{\lambda}}{ \alpha_0 s' +\beta_0}
$$
provided that
$$
x\, \alpha_0 +y\, \beta_0 \leq x\, \alpha_j+y\,  \beta_j+y\, \gamma_j, 
\quad \hbox{\rm  $j=1, \ldots, k$}.
$$
In our case, assuming that the above conditions hold, then we have
$$
{\rm slope}(p_*\phi^*D') =  
6\,+\, \frac{(31\, s'-132)k^2+(-39\, s'+186)\, k+5\, s'-24 }{ (3\, s'-12)\, k^3
+6\, k^2+(-5\, s'+20)\, k+s'-4}\, .
$$
From this one could deduce for even $g>4$ the following estimate for the moving slope
$$
\sigma(g) < 6 + \frac{20}{g} \; .
$$
For the reduced trace curve one can do similar things. The result
is the formula
$$
{\rm slope}(p_*\hat{\phi}^*D')=6+
{\frac { \left( 31\,s-132 \right) {k}^{2}+ \left( 264-63\,s \right) k-
36+8\,s}{ \left( 3\,s-12 \right) {k}^{3}+ \left( -2\,s+12 \right) {k}^
{2}+ \left( -9\,s+30 \right) k+2\,s-6}}
$$
and this results in a similar bound
$\sigma(g) < 6 + 20/g$. We refrain from giving details because, 
using the same Hurwitz space
but now as a correspondence between $\bM_g$ and $\bM_{0,6k}$ will result
in a better slope, as we show in the next section. 
\end{section}

\begin{section}{Another Correspondence}
The diagram
\begin{displaymath}
\begin{xy}
\xymatrix{
\bar{M}_{0,6k} &  \bar{H}_{2k,k+1} \ar[l]_{q} \ar[d]^{p} \\
& \bar{M}_{2k} \\
}
\end{xy}
\end{displaymath}
provides us with the action $p_*q^*$ on divisor classes. It is 
well-known that the divisor class 
$$
\kappa = \psi-\delta= \sum_{j=2}^{k} \frac{(j-1)(b-j-1)}{b-1} T_b^j
$$
is ample on $\bar{M}_{0,b}$. As above this gives us by $p_*q^*$
a moving divisor of good slope.
We  calculate now the class of $p_*q^* \kappa$.
With $\alpha(k,j)$ as defined in Theorem 1.1 of  \cite{GK} we get 
by combining relations (\ref{eq:pullbackT}) and the formulas in
Section \ref{pushforward} 
that
$$
p_*q^*T_b^{3j}= \sum_{c=0}^{[j/2]} (j+1-2c) e_{j,c} \, \delta_{j}= 
(6k)! \, \alpha(k,j) \; \delta_j \, .
$$
We also get
$$
p_*q^*\tb ^2= (6k)! \frac{b(b-1) N}{ 2\, (b-3)}\; [3(2k+5)\lambda 
- (k+1)\delta_0]  -(6k)! \sum_{j=1}^k (-c_j + \frac{9N}{4k-2} b_j) \; 
\delta_j \, .
$$
We therefore have
$$
\begin{aligned}
& p_*q^*\kappa =  \frac{b! bN}{2}\; [3(2k+5)\lambda - (k+1)\delta_0] \\
             & \quad - \frac{b!}{b-1}\; \sum_{j=1}^k [(b-3) \,(-c_j + 
\frac{9N}{4k-2}b_j) -  (3j-1)(b-3j-1) \, \alpha(k,j) ] \, \delta_j \, .
\end{aligned}
$$
\begin{theorem}
The moving slope $\sigma(g)$ of $\bar{M}_g$ for even $g$ satisfies 
the inequality
$$
\sigma(g) \leq 6+\frac{18}{g+2}\; .
$$
\end{theorem}
\begin{proof}
Indeed, if we write $p_*q^*\kappa$ as $a \lambda -\sum_{i=0}^k b_i \delta_i$
the ratio $a/b_0$ is $3(2k+5)/(k+1) = 6+18/(g+2)$, while $a/b_i$ for $i>0$
is much smaller as one sees by analyzing the expressions involved. 
\end{proof}

Observe also, that since $q^*\kappa $ is an ample class, all effective divisors in a multiple of this class 
intersect the positive dimensional fibers of the generically finite map $p$. We therefore conclude that the
common base locus $\cap _{m \geq 1} \mbox{Base}(|m \, (p_*q^*\kappa) |)$  is exactly the locus of points in $\bM_g$ 
over which the corresponding fiber of the map $p$ has positive dimension. It will be interesting to have a 
description of this common base locus.  

\end{section}

\begin{section}{The Prym Variety of the Trace Curve}
By associating to a point of $H_{2k,k+1}$ the Prym variety of $T/\hat{T}$
(resp.\ the quotient of 
the Jacobian of the reduced trace curve by the Jacobian of $C$)
we can define a morphism $\chi: H_{2k,k+1} \to {\mathcal A}_{(5k^2-k)/2}$, 
(resp. to $\hat{\chi}:  H_{2k,k+1} \to {\mathcal A}_{(5k-1)(k-2)/2}$), 
where ${\mathcal A}_n$ denotes a moduli space of polarized abelian 
varieties of dimension $n$.
The polarization is induced by the theta divisor on the Jacobian of the
trace curve. These maps are interesting and deserve further study.

Suppose that this map ${\chi}$ (resp.\ $\hat{\chi}$) 
extends to a rational map
$\chi: \tH \to \tilde{\mathcal A}$, a toroidal compactification that contains
the canonical rank $1$ partial compactification 
${\mathcal A}^{(1)}$ defined by Mumford.
Then the pullback under $\chi$ of the Hodge class is equal to 
$\phi^*{\lambda}_{g'}-\hat{\phi}^*{\lambda}_{\hat{g}}$. This 
expression is given by combining Theorems \ref{th: lambdag'} and 
\ref{th: lambdahatg}. Let  $D$ be the divisor that is the closure
of the inverse image of (open) boundary component of largest degree
under the map of $\tilde{\mathcal A}$ to the Satake compactification
${\mathcal A}^{\star}$. Let $L$ be the 
Hodge bundle (corresponding to modular forms of weight $1$)
Then the pull back of $D$ is given by
$\phi^*(\delta_0^{\prime}) -\hat{\phi}^*(\hat{\delta}_0)$. The Propositions
\ref{prop: pullbackdeltaprime} and \ref{prop: pullbackdeltahat} 
give expressions for this.  Thus we can calculate $p_*\chi^*(aL-bD)$
in terms of $\lambda, \delta_0$ and $\delta_j$ with $j=1,\ldots,k$.
Our expressions show that nef (ample) divisors $aL-bD$ with $a =12b$  
give rise to (moving) divisors of slope $6+20/g$. 
\end{section}

\begin{section}{The Eisenbud-Harris divisor}
The map $p: H_{2k,k+1} \to M_g$ is branched
along a divisor that was introduced and
studied by Eisenbud and Harris in \cite{EH1}. 
As a side product of our calculations
we now can calculate in an easy way 
the class of (the closure of) this divisor. 
We give only the coefficients of $\lambda $ and $\delta_0$
but the remaining coefficients can be calculated similarly. 

Since $\tilde{H}$ maps to $\bM_{0,b}$ and to $\bM_g$ via $q$ and $p$
we can calculate the canonical class in two ways:
$$
K_{\tilde{H}} = q^* K_{\bM_{0,b}} + R_q
\quad {\rm and} \quad
K_{\tilde{H}} = p^* K_{\bM_g} +R_p
$$
with $R_q$ and $R_p$ the ramification divisors. 
For $R_q$ we have, see relations (\ref{eq:pullbackT}), the formula
$$
R_q=E_2+2E_3 +\sum_{j,c} (j-2c) E_{j,c}\, ,
$$
while $R_p$ has four components, namely
$$
R_p=E_0+E_2+E_3+G,
$$
with $p_*G$ the Eisenbud-Harris divisor. Since we have formulas for
$p_*$ applied to the divisors $E_0,E_2,E_3$ and $E_{j,c}$ and we have a
formula for $p_*q^* K_{\bM_{0,b}}$ we can calculate $p_*G$.
Indeed, we get
$$
R_p= q^*K_{\bM_{0,b}}+R_q-p*K_{\bM_g}.
$$
Plugging in  the formula
$$
K_{\bM_{0,b}} = 
\frac{-2 }{ b-1} \, T_b^2 + \sum _{i=3}^{3k} (\frac{i(b-i)}{b-1}-2) \, T_b^i  
$$
and applying $p_*$ we find  $p_*G= p_*(R_p)-p_*(E_0+E_2+E_3)$ and thus get
$$
\begin{aligned}
p_*G = & \frac{-2 }{ b-1} \,p_*q^* T_b^2 +p_*E_3-p_*E_0 + \\
\quad & \sum_{j,c}[(\frac{3j(b-3j)}{b-1}-1)(j+1-2c)-1] p_*E_{j,c} - N_0 K_{\bM_g}
\\
\end{aligned}
$$
with $N_0=((6k)! N$. We now substitute
$K_{\bM_g}= 13 \lambda -2 \delta_0 -3\delta_1-2\sum_{j=2}^k \delta_j$
and find
$$
p_*G = \frac{N_0}{2k-1} [(6k^2+13k+1) \lambda - k(k+1) \delta _0 ] + \cdots
$$
in agreement with Theorem 2 of \cite{EH1}.
\end{section}
\vskip.1in
\noindent
{\bf Acknowledgement} The authors thank the University of 
Crete for supporting this work 
by the research grant no. 3215 of the Program for Funding Basic Research.

 \end{document}